\documentclass[12pt,reqno]{amsart}
\usepackage{amsfonts}
\usepackage{pdfsync}
\usepackage{amsfonts} 
\usepackage{amssymb, amsmath, mathrsfs}
 \usepackage{graphicx,graphics,color,epsfig,setspace}
\newcommand{\nc}{\newcommand}
\nc{\parent}[1]{$[\![#1]\!]$}
\newtheorem{theorem}{Theorem}[section]
\newtheorem{lemma}{Lemma}[section]
\newtheorem{example}{Example}[section]
\newtheorem{corollary}{Corollary}[section]
\newtheorem{proposition}{Proposition}[section]
\newtheorem{remark}{Remark}[section]
\newtheorem{definition}{Definition}[section]
\newtheorem{assumption}{Assumption}[section]

\newenvironment{pf-main}{{\bf \sc Proof of Theorem \ref{mainresult}.}\hspace{3mm}}{\qed}
\nc{\bfE}{\mathbf{E}} \nc{\bfT}{\mathbf{T}}\nc{\sE}{{\mathscr
E}} 
\nc{\cadlag}{c\`{a}dl\`{a}g } \nc{\ba}{\begin{array}}
\nc{\ea}{\end{array}} \nc{\be}{\begin{equation}}
\nc{\ee}{\end{equation}} \nc{\bea}{\begin{eqnarray}}
\nc{\eea}{\end{eqnarray}} \nc{\bean}{\begin{eqnarray*}}
\nc{\eean}{\end{eqnarray*}} \nc{\bu}{\bullet} \nc{\nn}{\nonumber} \nc{\sB}{{\mathscr B}}
\nc{\cA}{{\mathcal A}} \nc{\cB}{{\mathcal B}} \nc{\cC}{{\mathcal C}}
\nc{\cD}{{\mathcal D}} \nc{\bbD}{\mathbb{D}}\nc{\bbH}{\mathbb{H}}
\nc{\bbF}{\mathbb{F}}\nc{\bbG}{\mathbb{G}}\nc{\cG}{{\mathcal G}} \nc{\cF}{{\mathcal F}}
\nc{\cS}{{\mathcal S}} \nc{\cU}{{\mathcal U}} \nc{\cH}{{\mathcal H}}
\nc{\cK}{{\mathcal K}} \nc{\cL}{{\mathcal L}} \nc{\cM}{{\mathcal M}}
\nc{\cO}{{\mathcal O}} \nc{\cP}{{\mathcal P}} \nc{\bbE}{\mathbb{E}}
\nc{\bbEQ}{\mathbb{E}_{\mathbb{Q}}} \nc{\eps}{\varepsilon}
\nc{\bbEP}{\mathbb{E}_{\mathbb{P}}}\nc{\bbL}{\mathbb{L}}
\nc{\bbP}{\mathbb{P}} \nc{\bbQ}{\mathbb{Q}} \nc{\del}{\partial}
\nc{\Om}{\Omega} \nc{\om}{\omega} \nc{\bbR}{\mathbb{R}}
\nc{\bbC}{\mathbb{C}} \nc{\bfr}{\begin{flushright}}
\nc{\efr}{\end{flushright}} \nc{\dXt}{\Delta X_{t}} \nc{\dXs}{\Delta
X_{s}} \nc{\bs}{\blacksquare} \nc{\dX}{\Delta X} \nc{\dY}{\Delta Y}
\nc{\dnkx}{\left(X(T^{n}_{k})-X(T^{n}_{k-1})\right)}
\nc{\esssup}{\mathrm{ess}\mbox{ }\mathrm{sup}}
\nc{\essinf}{\mathrm{ess}\mbox{ } \mathrm{inf}}
\nc{\dhats}{\widehat{\delta_s}} \nc{\tX}{\tilde{X}}
\nc{\tZ}{\tilde{Z}}
\nc{\what}{\widehat}
 \nc{\half}{\frac{1}{2}}
 \nc{\wtilde}{\widetilde}
\def\rar{\rightarrow}
\nc{\uar}{\uparrow}
\nc{\dar}{\downarrow}

\nc{\chf}{\mbox{$\mathbf1$}} \nc{\eid}{\stackrel{d}{=}}

\DeclareFontFamily{U}{mathx}{\hyphenchar\font45}
\DeclareFontShape{U}{mathx}{m}{n}{
      <5> <6> <7> <8> <9> <10>
      <10.95> <12> <14.4> <17.28> <20.74> <24.88>
      mathx10
      }{}
\DeclareSymbolFont{mathx}{U}{mathx}{m}{n}
\DeclareMathSymbol{\bigtimes}{1}{mathx}{"91}
\begin{document}
\title{Markov bridges: SDE representation}
\author[]{Umut \c{C}etin}
\address{Department of Statistics, London School of Economics and Political Science, 10 Houghton st, London, WC2A 2AE, UK}
\email{u.cetin@lse.ac.uk}
\author[]{Albina Danilova}
\address{Department of Mathematics, London School of Economics and Political Science, 10 Houghton st, London, WC2A 2AE, UK}
\email{a.danilova@lse.ac.uk}

\begin{abstract} Let $X$ be a Markov process taking values in $\bfE$  with continuous paths and transition function $(P_{s,t})$. Given a measure $\mu$ on $(\bfE, \sE)$, a Markov bridge starting at $(s,\eps_x)$ and ending at $(T^*,\mu)$ for $T^* <\infty$ has the law of the original process starting at $x$ at time $s$ and  conditioned  to have law $\mu$ at time $T^*$. We will consider two types of conditioning: a) {\em weak conditioning} when  $\mu$  is absolutely continuous with respect to $P_{s,t}(x,\cdot)$  and b) {\em strong conditioning} when  $\mu=\eps_z$ for some $z \in \bfE$.   The main result of this paper is the representation of a Markov bridge as a solution to a stochastic differential equation (SDE)  driven by a Brownian motion in a diffusion setting. Under mild conditions on the transition density of the underlying diffusion process we establish the existence and uniqueness of  weak and strong solutions of this SDE. 
 \end{abstract}
\date{\today}
\maketitle
\section{Introduction}
The main purpose of this paper is to study path-wise construction of a Markov process on $[0,T^*)$ , where $T^* \in (0,\infty]$, starting at $x$ and arriving at $z$ at $T^*$.  A canonical example of such a process is the Brownian bridge on $[0,1)$:
\be \label{e:bb}
X_t=x+ B_t -t B_1 +(z-x)t, \qquad t \in [0,1).
\ee
If one defines $(\beta_t)_{t \in [0,1)}$ by 
\[
d\beta_t= dB_t -\frac{B_1-B_t}{1-t}dt,
\]
then $\beta$ becomes a Brownian motion in the natural filtration of $B$ initially enlarged with $B_1$. Moreover, $X$ solves the following SDE:
\be \label{smb:e:BB}
X_t= x+ \beta_t +\int_0^t \frac{z-X_s}{1-s}ds, \qquad t \in [0,1).
\ee
Conversely, if one starts with an arbitrary Brownian motion, $\beta$, the solution to the above SDE has the same law as the Brownian bridge defined by (\ref{e:bb}). In particular, $\lim_{t\rar 1} X_t=z$, a.s.. Both (\ref{e:bb}) and (\ref{smb:e:BB}) provide a path-wise construction of a Brownian bridge from a given Brownian motion. The crucial difference is that while the former construction is not adapted to the filtration of the given Brownian motion, the latter is.   

In this paper we will study analogous conditionings for a class of continuous Markov processes taking values in $\bbR^d$. It is known that this problem has a solution in the case of Brownian and Bessel bridges, which have been studied extensively in the literature and found numerous applications (see, e.g., \cite{PY81}, \cite{PY82},  \cite{BeP},  \cite{Hsu}, \cite{Sa}, and \cite{SW}). For a general right continuous strong Markov process \cite{FPY} constructs a measure on the canonical space such that the coordinate process have the prescribed conditioning under a duality hypothesis. More recently, \cite{C-UB-MB} performed the same construction without the duality hypothesis under the assumption that the semigroup $(P_t)$ of the given process has continuous transition densities and $\|P_t-I\| \rar 0$ as $t \rar 0$, where $\|\cdot\|$ corresponds to the operator norm.  Moreover, they have proven that if the original process is, in addition, self-similar, then a  path-wise construction of the Markov bridge can be performed. However, this construction is not adapted. 

In what follows  we will show that given an $\bbR^d$-valued diffusion, $Y$, and a Brownian motion, $B$, one can construct a Markov bridge which is adapted to the natural filtration of $B$ as a strong solution of the following SDE:
\be \label{e:i:sde}
X_t=x+\int_s^t \left\{ b(X_u)+ a(X_u)\frac{(\nabla h^z(u,X_u))^*}{h^z(u,X_u)}\right\}du+\int_s^t \sigma(X_u)\,dB_u,
\ee
where $h^z(t,x)=p(T^*-t,x,z)$, and $p$ is the transition density of $Y$. In the above representation $\sigma$ and $b$ are the diffusion and drift coefficients of $Y$ and $a=\sigma \sigma^*$. Although this SDE can be obtained via a formal application of an $h$-transform, the proof of existence and uniqueness of a strong solution is problematic due to the explosive behaviour of its coefficients. 

The SDE (\ref{e:i:sde}) resembles the Doob-Meyer decomposition of the process $Y$ defined on some probability space $(\Om, \cF, \bbP)$ with respect to its natural filtration enlarged with the value of $Y_{T^*}$. Indeed, it is well-known (see \cite{JY}, \cite{ThJ}, \cite{MY}) that in this enlarged filtration $Y$ satisfies
\[
Y_t=y+\int_0^t \left\{ b(Y_u)+ a(Y_u)\frac{(\nabla h^{Y_{T^*}}(u,Y_u))^*}{h^{Y_{T^*}}(u,Y_u)}\right\}du+\int_s^t \sigma(Y_u)\,d\beta_u,
\]
where $\beta$ is a Brownian motion under the enlarged filtration independent of $Y_{T^*}$ provided
\[
\int_0^t \left| b(Y_u)+ a(Y_u)\frac{(\nabla h^{Y_{T^*}}(u,Y_u))^*}{h^{Y_{T^*}}(u,Y_u)}\right|du <\infty, \, \bbP-\mbox{a.s. for each } t<T^*.
\]

If $\Om$ is a complete and separable metric space and $\cF$ is the collection of its Borel sets, e.g. if $(\Om, \cF, P)$ is the Wiener space, then there exists a family of regular conditional probabilities  $Q^y(z; \cdot)$ on $\cF$ given $\sigma(Y_{T^*})$ such that for $\mu_y$-a.e. $z$ one has $Q^y(z; E)=\bbP(E|Y_{T^*}=z)$  and $Q^y(z; [Y_{T^*}=z])=1$,  where $\mu_y$ is the law of $Y_{T^*}$ (see Theorem 5.3.19 in \cite{KS}).  Since $Y_{T^*}$ and $\beta$ are independent under $\bbP$, they will remain so under $Q^y(z; \cdot)$. This implies that for a given $x$ there exists a weak solution to (\ref{e:i:sde}) for $\mu_x$-a.e. $z$. One can extend this existence result to {\em all} $z$ if, e.g, both $Q^y(z; \cdot)$ and $\bbP(\cdot|Y_{T^*}=z)$ are continuous in $z$, which is difficult to verify in general. 

Note that the set, $E_x$, of $z$ for which there exists a weak solution to (\ref{e:i:sde}) depends on $x$.  Thus, for a given uncountable Borel set $S$, the set of  $z$ for which there exists a weak solution to (\ref{e:i:sde}) for {\em all} $x \in S$, i.e.  $\cap_{x \in S}E_x$, might be a null set. Indeed, suppose that $\mu_x \sim m$ for all $x \in E$, where $m$ is a measure on the Borel subsets of $\bbR^d$ without point mass. Then, $m(\cap_{x \in S}E_x)$ might be less than $m(\bbR^d)$ and, in particular, could be equal to $0$. 

The existence of $z$ such that there exists a solution to (\ref{e:i:sde}) for all $x$ is important if one wants to establish the strong Markov property of the solutions of (\ref{e:i:sde}) or, equivalently, the well-posedness of the associated local martingale problem. In view of the discussion above  an approach based on enlargement of filtrations cannot deliver the strong Markov property of the solutions.  This shortcoming is remedied by the techniques developed in this paper that  lead to  the existence and uniqueness of strong solutions of this SDE for {\em all} pairs $(x,z)$. In particular, we demonstrate the strong Markov property of the solutions of the SDE. 
 
The standard approach to establish a unique strong solution, which we follow in this paper,  is via the result due to Yamada-Watanabe which requires the existence of a weak solution and the path-wise uniqueness. As $h$ is not bounded from below, the standard results on pathwise uniqueness are not applicable but this issue can be circumvented by localisation arguments. The existence of a weak solution, however, is a much more delicate problem to handle.  One can relatively easily construct a weak solution on $[0,T]$ for any $T<T^*$, thus obtaining a sequence of consistent measures, $Q^T$. At this point one is tempted to use Kolmogorov's extension theorem to find a measure $P$ which solves the corresponding martingale problem on $[0,T]$ for any $T<T^*$. The issue with this argument is that the resulting measure is not necessarily concentrated on the paths with left limits. 

This issue can be resolved if one is willing to assume the existence of a right continuous and strong Markov dual process. The idea, as observed in \cite{FPY}, is to construct a dual process that starts at $z$ and is conditioned to arrive at $y$ at $T^*$. The law of this process is the image of the solution to (\ref{e:i:sde}) under time reversal, which gives the desired property of the measure obtained via Kolmogorov's extension, as well as   the bridge property, i.e. $P(\lim_{t \rar T^*}X_t=z)=1$.  We, on the other hand, use weak convergence techniques to establish the existence of such a measure.  This construction does not rely on neither duality assumptions nor self-similarity of $Y$. 

The strong solutions of (\ref{e:i:sde}) are also related to the question of time reversal of diffusion processes. In particular, if $\sigma \equiv 1$ it is known that the time reversed process, $(X_{T^*-t})$ satisfies the above SDE weakly on $[0,T^*)$ under some mild conditions. The SDE representation for the reversed process was obtained by F\"ollmer in \cite{Follmer} using entropy methods in both Markovian and non-Markovian case, and by Hausmann and Pardoux \cite{HP} via weak solutions of backward and forward Kolmogorov equations. Later Millet et al. \cite{MNS} extended the results of Hausmann and Pardoux by means of Malliavin calculus  to obtain the necessary and sufficient conditions for the reversibility of diffusion property. This problem was also tackled with the enlargement of filtration techniques by Pardoux \cite{PEF}. 

We will refer the type of conditioning represented by (\ref{e:i:sde}) as a {\em strong conditioning}. In this paper we will also consider weak conditioning when the original process is conditioned to have a given distribution at $T^*$, which is absolutely continuous with respect to its original distribution at time $T^*$. In contrast with strong conditioning this construction is based on a careful implementation of Kolmogorov's extension argument. The weak conditioning that we consider can be viewed as a generalisation of the result of Jamison \cite{Jamison} that studies the Markov processes related to a problem posed by Schroedinger. One can also give an interpretation of weak conditioning when $T^*=\infty$ in the context of penalisations on the canonical space (see \cite{RVY} and \cite{RoyYor} for a review of the topic).  

The rest of the paper is organised as follows. Section 2 contains the main results together with their discussion and examples.  The proofs are postponed until Section 3. 

\noindent {\bf Acknowledgement:} The authors would like to thank the anonymous referee whose remarks and comments have significantly improved the paper.  

\section{Main results and examples}

Let $T^*\in (0, \infty]$. If $T^*<\infty$, we suppose $\Om=C([0,T^*], \bfE)$ where $\bfE=\bigtimes_{i=1}^d [l_i, \infty)$ with the convention that if $l_i=-\infty$ then $[l_i,\infty)=\bbR$.  In case $T^{*}=\infty$, $\Om=C([0,\infty), \bfE)$. We equip $\bfE$ with the metric $\rho$ defined by $\rho(x,y)=\|x-y\|$, where $\|\cdot\|$ corresponds to the Euclidean norm on $\bbR^d$, and $\sE$ will stand for the set of all Borel subsets of $\bfE$. We denote by $\bfT$ the index set of time, i.e.  $\bfT=[0,T^*]$ (resp.  $\bfT=[0,\infty)$) when  $T^*<\infty$ (resp. $T^*=\infty$). Similarly, $\bfT_s= [s,\infty)\cap \bfT$.  We endow $\Om$ with the local uniform topology so that it is a Polish space and denote by $X$ the coordinate process. The canonical filtration $(\cB_t)_{t \in [0, T^*)}$ is defined via  $\cB_t=\sigma(X_s; s\leq t)$ for $t<T^*$, and $\cB_{T^*}=\vee_{t <T^*}\cB_t$.  
%We will write $\cB_{t \geq s}$ for $(\cB_t)_{t \in [s, T^*]}$ (resp. $(\cB_t)_{t \in [s, T^*)}$) when $T^*<\infty$ (resp. $T^*=\infty$).

Let $A$ be the generator defined by
\be \label{e:At}
A_t=\frac{1}{2}\sum_{i,j=1}^d a_{ij}(t,\cdot)\frac{\partial^2}{\partial x_i \partial x_j}+\sum_{i=1}^d b_i(t,\cdot)\frac{\partial }{\partial x_i},
\ee
where $a$ is a matrix field and $b$ is a vector field. We suppose the following.
\begin{assumption} \label{smb:a:A}
\begin{enumerate}
\item For all $i,j=1, \ldots, d$, the maps $(t,x) \mapsto a_{ij}(t,x)$ and $(t,x) \mapsto b_i(t,x)$ are real-valued and Borel measurable, and $a_{ij}$ is locally bounded. Moreover, either $b_i$ is locally bounded from above for all $i$ or locally bounded from below for all $i$.
\item For each $(t,x)$ the matrix $a(t,x)$ is symmetric and non-negative, i.e. for every $\lambda \in \bbR^d$
\[
\sum_{i,j} a_{ij}(t,x)\lambda_i\lambda_j \geq 0.\]
\item \label{a:Awp} The local martingale problem for $A$ where solutions have sample paths in $\Om$ is well-posed, i.e. for any $(s,x) \in [0,T^*) \times \bfE$, there exists a unique probability measure $P^{s,x}$ on $(\Om, \cB_T^{*})$ such that $P^{s,x}(X_r=x, r \leq s)=1$ and $\left((M^f_t)_{t \in \bfT_s}, (\cB_t)_{t \in \bfT_s}\right)$ is a local martingale, where
\[
M^f_t=f(X_t)-f(X_s) -\int_s^t A_r f(X_r)\,dr,
\]
for any $f \in C^{\infty}(\bfE)$.
\end{enumerate}
\end{assumption}
Well-posedness of the local martingale problem described in Part \ref{a:Awp} of Assumption \ref{smb:a:A} implies that $X$ is strong Markov under $P^{s,x}$ (Theorem 4.4.2 in \cite{EK}) and we will denote its transition function with $(P_{r,t})$.
%and for every $s <T^*$ the mapping $x \mapsto P^{s,x}$ is Borel measurable (Theorem 4.4.6 in \cite{EK}). Thus, $P^{s,\mu}:=\int_{\bfE}P^{s,x}\mu(dx)$ is well-defined for any  probability measure, $\mu$,  on $(\bfE, \sE)$ and it solves the local martingale problem for $A$ with $P^{s,\mu}(X_r=X_s, r \leq s)=1$ and $P^{s,\mu}X_s^{-1}=\mu$. 
Moreover,  we show in  Appendix \ref{s:app} that  well-posedness of the local martingale problem is equivalent to the existence and uniqueness in law of a weak solution for an associated stochastic differential equation. 

This relationship between the local martingale problem and the weak solutions of SDEs provides a generic approach to the well-posedness of the local martingale via a study of the associated SDE. In particular, the local martingale problem is well-posed if (see Remark 5.4.30 and Corollary 5.4.29 in \cite{KS}) $\bfE=\bbR^d$,  the coefficients $a_{ij}$ and $b_i$ are H\"{o}lder continuous and bounded,  and the matrix $a$ is uniformly positive definite, i.e.
\[
\sum_{i,j} a_{ij}(t,x)\lambda_i\lambda_j \geq c \|\lambda\|^2, \qquad\forall x, \lambda \in \bbR^d \mbox{ for some } c>0.\] 

Given a measure $\mu$ on $(\bfE, \sE)$, a Markov bridge starting at $(s,\eps_x)$ and ending at $(T^*,\mu)$ for $T^* <\infty$ has the law of the original process, $X$, given that $X_{T^*}$ has the law $\mu$. We will consider two types of conditioning: a) {\em weak conditioning} when  $\mu$  is absolutely continuous with respect to $P_{s,T^*}(x,\cdot)$  and b) {\em strong conditioning} when  $\mu=\eps_z$ for some $z \in \bfE$.  

The weak conditioning can be obtained via an {\em $h$-function}. Indeed, since  $\mu$ is   absolutely continuous with respect to   $P_{s,T^*}(x,\cdot)$, there exists a Radon-Nikodym derivative, $H$, so that
\[
\mu(E)=\int_E H(y) P_{s,T^*}(x,dy), \forall E\in \sE.
\]
If we define the function $h:[0,T^*]\times \bfE \mapsto \bbR_+$ by
\[
h(t,y):=\int H(z) P_{t,T^*}(y,dz),
\]
then $\left((h(t,X_t)_{t \in \bfT_s},  (\cB_t)_{t \in \bfT_s}\right)$ is a martingale under  $P^{s,x}$.  The weak conditioning can be obtained via Theorem \ref{c:hsde} under some technical conditions listed in the definition below. 
\begin{definition} For $T^*<\infty$ (resp. $T^*=\infty$), we call a function $h:[0, T^*]\times \bfE \mapsto [0,\infty)$ (resp.  $h:[0, T^*) \times \bfE \mapsto (0,\infty)$), strictly positive on $[0,T^*)\times \bfE$, an {\em $h$-function} if $h \in C^{1,2}([0,T^*)\times\bfE)$ and $\left((h(t,X_t)_{t \in \bfT},  (\cB_t)_{t \in \bfT}\right)$ is a martingale under every $P^{0,x}$. 
\end{definition}

\begin{theorem}  \label{c:hsde} Let $\sigma$ be a matrix field such that $a=\sigma\sigma^*$. Suppose that the conditions of Assumption \ref{smb:a:A} hold and let $h$ be an $h$-function. Then, there exits a unique weak solution on $\bfT_s$ to the following SDE:
\[
X_t=x+\int_s^t \left\{ b(u,X_u)+ a(u,X_u)\frac{(\nabla h(u,X_u))^*}{h(u,X_u)}\right\}du+\int_s^t \sigma(u,X_u)\,dB_u.
\]
$X$ is a strong Markov process  and the associated transition function, $(P^h_{r,t})$ is related to $(P_{r,t})$ via
\[
P^h_{r,t}(x,E)=\frac{1}{h(r,x)} \int_E h(t,y)P_{r,t}(x,dy), \qquad x \in \bfE, \, E\in \sE, s\leq r<t, t  \in \bfT_s.
\]
In particular, if $T^* <\infty$, 
\be \label{smb:e:limitdist} P(X_{T^*}\in E)=\int_E \frac{h(T^*,y)}{h(s,x)}P_{s,T^*}(x,dy).
\ee
\end{theorem}

Observe that the definition of an $h$-function as well as Theorem \ref{c:hsde} does not require $T^*<\infty$.  
%If the original process $X$ has a limiting distribution, $\nu$, and $h$ is an $h$-function that depends only on $x$, the solution of the SDE in the above theorem has a limiting distribution given by $h(x)\nu(dx)$, i.e. the solution of this SDE can be interpreted as a bridge of infinite length with a new limiting distribution. This is in spirit of (\ref{smb:e:limitdist}) once we replace $T^*$ with $\infty$ and interpret $P_{s,T^*}(x,dy)$ as the limiting distribution of the original process. 
A way to interpret the weak conditioning with $T^*=\infty$ is via penalisations (see, e.g., \cite{RVY} and \cite{RoyYor} for the theory and examples). Indeed, in this case the law of the bridge process in Theorem \ref{c:hsde}  can be viewed as the penalised probability measure on $(\Om, \cB_{\infty})$ induced by the weight process $(h(t,X_t))_{t \in \bfT}$. Thus, weak conditioning with $T^*=\infty$ is an example of penalisation when the weight process is an adapted martingale. 

\begin{example} \label{smb:x:h1}Let  $T^*<\infty$ and suppose that  $E \in \sE$ is a set such that $(t,x)\mapsto P_{t,T^*}(x,E)$ belongs to $C^{1,2}([0,T^*)\times \bfE)$ and $P_{t,T^*}(x,E)>0$ for all $t<T^*$ and $x \in \bfE$. Define $h$ by $h(T^*,x)=\chf_E$ and $h(t,x)=P_{t,T^*}(x,E)$. Clearly, $\left((h(t,X_t)_{t \in \bfT},  (\cB_t)_{t \in \bfT}\right)$ is a bounded martingale under every $P^{0,x}$. Moreover,  $h \in C^{1,2}([0,T^*)\times\bfE)$ by assumption. Thus, $h$ is an $h$-function. If we apply the above theorem to this $h$-function, we end up with a weak conditioning of the coordinate process that ensures that $P^{h;s,x}(X_{T^*} \in A)=1$. 
\end{example}
%\begin{example} Consider $d=1$, $\bfE=\bbR$, $a(x)=1$, and $b(x)=-\lambda x$ for some $\lambda>0$. Then, $X$ is an Ornstein-Uhlenbeck with a limiting distribution given by a Normal random variable with mean $)$ and variance $(2 \lambda)^{-1}$. Choose
%\[
%h(x)=\int_0^x e^{\lambda z^2}dz,
%\]
%and observe that $h(X_t)$ is a $P^x$-martingale for any $x \in \bbR$. Thus, Theorem \ref{c:hsde} yields the existence of a unique weak solution to 
%\[
%X_t= x+ \int_0^t \left\{-\lambda X_s +\frac{h'(X_s)}{h(X_s)}\right\}ds + B_t.
%\]
%Moreover, $X$ has a limiting distribution with density
%\[
%\frac{h(x) e^{- \lambda x^2}}{\int_{-\infty}^{\infty}h(y) e^{- \lambda y^2}dy}.
%\]
%\end{example}

 Strong conditioning, intuitively, can be done via ``$h(t,x)=P^{t,x}(X_{T^*}=z)$''. As $[X_{T^*}=z]$ is most likely a null set, the above theorem is not applicable since the $h$-function vanishes. We will obtain the stochastic differential equation for the bridge process under two different sets of conditions. The first set of conditions will be handy when one can obtain bounds on the transition density of the process associated with the  solution of the local martingale problem over the interval $[0,T^*]$, e.g. via Gaussian type estimates on the fundamental solution of the parabolic pde $u_t = Au$. Although this assumption is stated for a generator with time-independent coefficients, its generalisation to the time-dependent case is straightforward and the proof will hold verbatim with the obvious modifications. We demonstrate the proof for the time-independent case for the sake of brevity of exposition. 

The second set of assumptions can be seen as  a relaxation of the first one in the case of a time-homogeneous local martingale problem whose solution has sample paths in $C((0,\infty), \bfE)$.  The proof in the latter case relies on a certain bounded property of the potential density of $X$, which is generally satisfied in the one-dimensional case (see Proposition \ref{smb:l:ubdd}). 
\begin{assumption} \label{smb:a:pbdd} Suppose that $a$ and $b$ do not depend on time, $\bfT=[0,T^*]$, Assumption \ref{smb:a:A} is satisfied, and  the family of solutions of the local martingale problem, $P^x$, is weakly continuous. Moreover, $({P}_t)$ is a  semi-group admitting a regular transition density $p(t,x,y)$ with respect to a $\sigma$-finite  measure $m$ on $(\bfE, \sE)$ such that 
\be \label{smb:e:dual}
\lim_{t \rar 0} \int_{B_r^c(z)} p(t,y,z)p(u-t,x,y)m(dy)=0, \; \forall u >0,\, r>0,
\ee
where $B_r(z):=\{y:\|y-z\|<r\}$, the Chapman-Kolmogorov equations,
\be \label{smb:e:C-K}
p(t, x,y)=\int_{\bfE} p(t-s,x,u)p(s,u,y)m(du), \; 0<s<t\leq T^*, 
\ee
hold, and for every $z \in \bfE$ and $r>0$ 
\be \label{smb:e:tdbound}
\sup_{\substack{x \notin B_r(z)\\ t \leq T^*}}p(t, x,z) < \infty.
\ee
 \end{assumption}
 
The condition (\ref{smb:e:dual}) is satisfied when $p(s,x,y)$ is continuous on $(0,T^*)\times \bfE\times \bfE$ and there exists a right-continuous process, $\tilde{X}$ such that $X$ and $\tilde{X}$ are in duality with respect to $m$. This would be the case if $X$ were a strongly symmetric Borel right process (see Remark 3.3.5 in \cite{MarRos}) with continuous transition densities, in particular a one-dimensional regular diffusion without an absorbing boundary. Moreover, if $X$ is a Feller process the laws $(P^x)$ will be weakly continuous, too.

The boundedness assumption on the transition density as given in (\ref{smb:e:tdbound}) is satisfied in many practical applications. In particular, if the coefficients $b_i$ and $a_{ij}$ are bounded, H\"older continuous, and the matrix $a$ is uniformly positive definite, then $m$ becomes the Lebesgue measure and  $p(t,x,y)$ becomes the fundamental solution of the parabolic PDE, $u_t=Au$, and satisfies for some $k>0$
\be \label{smb:e:fud}
p(t,x,y)\leq t^{-\frac{d}{2}}\exp\left(-k\frac{\|x-y\|^2}{t}\right), \qquad t\leq T^*,
\ee
yielding the desired boundedness (see Theorem 11 in Chap. 1 of \cite{friedman}). Also observe that this estimate implies (\ref{smb:e:dual}), too. Moreover, $P_tf$ is a continuous function vanishing at infinity whenever $f$ is continuous and vanishes at infinity, i.e., $X$ is Feller. 
 \begin{assumption} \label{smb:a:potden} Suppose that $a$ and $b$ do not depend on time, $\bfT=[0,\infty)$, Assumption \ref{smb:a:A} is satisfied, and  the family of solutions of the local martingale problem, $P^x$, is weakly continuous. Moreover, $({P}_t)$ is a  semi-group admitting a regular transition density $p(t,x,y)$ with respect to a $\sigma$-finite measure $m$ on $(\bfE, \sE)$ such that 
\be \label{smb:e:dualinf}
\lim_{t \rar 0} \int_{B_r^c(z)} p(t,y,z)p(u-t,x,y)m(dy)=0, \; \forall u >0,\, r>0,
\ee
and the Chapman-Kolmogorov equations,
\be \label{smb:e:C-Kinf}
p(t, x,y)=\int_{\bfE} p(t-s,x,u)p(s,u,y)m(du), \; 0<s<t,
\ee
hold. 

Furthermore,  the $\alpha$-potential density\footnote{ The $\alpha$-potential density defines the kernel of the $\alpha$-potential operator. That is, for any $f \in C_b(\bfE)$
 \[
 U^{\alpha}f(x):=\int_0^{\infty} e^{-\alpha t}P_tf(x)dt=\int_{\bfE}u^{\alpha}(x,y) f(y)m(dy).
 \]}, $u^{\alpha}$, defined by
\[
u^{\alpha}(x,y) := \int_0^{\infty} e^{-\alpha t}p(t,x,y)dt, \qquad (x,y) \in \bfE\times \bfE,
\]
satisfies 
\be \label{smb:e:bddpotential}
\sup_{\alpha >0, x \in K} \alpha u^{\alpha}(x,y) <\infty
\ee
for any $y \in \bfE$ and a compact set $K \subset \bfE$ such that $y \notin K$.
 \end{assumption}
 
\begin{theorem} \label{smb:th:MB} Let $\sigma$ be a matrix field such that $\sigma \sigma^*=a$.  Suppose that Assumption \ref{smb:a:pbdd} or \ref{smb:a:potden} is in force, $T^*<\infty$. Fix $x \in \bfE$ and $z \in \bfE$ such that the following conditions hold:
\begin{enumerate}
\item $m(\{z\})=0$ and $p(T^*,x,z)>0$.  
\item For $h(t,y)=p(T^*-t,y,z)$ either $h \in C^{1,2}([0,T^*), \bfE)$ or $h \in C^{1,2}([0,T^*), \mbox{int}(\bfE))$, $x \in \mbox{int}(\bfE)$, and $ P_t(x,\mbox{int}(\bfE))=1 $ for all $t\leq T^*$.
\item If Assumption \ref{smb:a:potden} is enforced, then
\begin{itemize}
\item[i)] $u^{\alpha}(x,z)<\infty$ for $\alpha>0$,
\item[ii)] either the map $t\mapsto p(t,x,y)$ is continuous on $(0, \infty)$ for every $y \in \bfE$, or for all $t>0$ $p(t,x,y)>0, \,m$-a.e. $y$. 
\end{itemize}
\end{enumerate}
Then  there exists a  weak solution  on $[0,T^*]$  to 
\be \label{smb:e:sdeB}
X_t=x+\int_0^t \left\{ b(X_u)+ a(X_u)\frac{(\nabla h(u,X_u))^*}{h(u,X_u)}\right\}du+\int_0^t \sigma(X_u)\,dB_u,
\ee
the law of which, $P^{x \rar z}_{0 \rar T^*}$, satisfies  $P^{x \rar z}_{0 \rar T^*} (\inf_{u \in [0,T]}h(u,X_u) =0)=0$ for any $T<T^*$, and $P^{x \rar z}_{0 \rar T^*}(X_{T^*}=z)=1$. 

In addition, if $h(t, \cdot)>0$ for all $t<T^*$, weak uniqueness holds for the above SDE. Moreover, $X$ is a  Markov process with transition function $(P^h_{s,t})$ defined by
\[
P^h_{s,t}(x,E)=\int_{E} \frac{p(t-s,x,y)p(T^*-t,y,z)}{p(T^*-s,x,z)}m(dy), s<t<T^*, x\in \bfE, E\in \sE.
\]
\end{theorem}
\begin{remark} Condition $m(\{z\})=0$ is in fact not necessary. Indeed, if $m(\{z\})>0$, then $P^x(X_{T^*}=z)>0$ due to $p(T^*,x,z)>0$. This implies that we are in the setting of Example \ref{smb:x:h1} and, therefore, Theorem \ref{c:hsde} is applicable. If one, however, still wants to use the weak convergence techniques employed in the proof of the above theorem, one can do so without the convergence result of Lemma \ref{smb:l:potential} since $M$ of the lemma is bounded by $1/m(\{z\})$ and (\ref{smb:e:est2.1}) follows from (\ref{smb:e:est2.1T}) by the Dominated Convergence Theorem. Also note that whenever $m(\{z\})>0$, we do not need (\ref{smb:e:dual}) or (\ref{smb:e:dualinf}) to complete the proof either. Moreover, both (\ref{smb:e:tdbound}) and (\ref{smb:e:bddpotential}) are automatically satisfied.
\end{remark}

One can in fact obtain a unique strong solution to (\ref{smb:e:sdeB}) under slightly stronger conditions on the coefficients and the transition density (see Theorem \ref{smb:th:sunique}). Moreover, if the conditions of Theorem \ref{smb:th:MB} are satisfied by all $x \in \bfE$ and $h(t, \cdot)>0$ for all $t<T^*$, then $X$ is strong Markov (see Corollary \ref{smb:c:smarkov}). 

We end the discussion of main results with the following examples of strong conditioning of one-dimensional diffusions and Gaussian processes.
\begin{example} \label{ex:1dd} Consider the case  $ \bfE= [l, \infty)$ and $a$ and $b$ do not depend on time.  Moreover, suppose that $a$ and $b$ are continuous on $(l, \infty)$, $a$ is strictly positive on $(l,\infty)$, and
\[
\forall x \in (l, \infty), \exists \eps>0 \mbox{ s.t. } \int_{x-\eps}^{x+\eps}\frac{|b(y)|}{a(y)}dy <\infty.
\]
Let $c \in (l,\infty)$ be an arbitrary point and define the {\em scale function}
\[
s(x):=\int_c^x\exp\left(-2\int_c^y \frac{b(z)}{a(z)}dz\right)dy.
\]
Note that under the above  assumptions $s$ is twice continuously differentiable on $(l, \infty)$ and the first derivative is strictly positive. 

The associated {\em speed measure}, $m$, on $(l, \infty)$ is characterised by
\[
m(dx)= \frac{2}{a(x)s'(x)}dx.
\]
We further assume that the endpoints of $\bfE$ are inaccessible, that is, 
\be \label{smb:e:iab}
\int_c^{\infty}m((c,x))s'(x)dx =\int_l^c m((x,c))s'(x)dx =\infty.
\ee
This yields in view of Theorem 5.5.15 in \cite{KS} that there exists a unique weak solution to
\[
dX_t=\sigma(X_t) dB_t +b(X_t)dt,
\]
and consequently the local martingale problem for $A$ is well-posed by Theorem \ref{g:t:lm2ws}.

We will in fact require more and assume that the infinite boundaries are natural. This  ensures that $(P^x)$ is Feller (see Theorem 8.1.1 in \cite{EK}), and, therefore weakly continuous.

Since the end-points are inaccessible McKean \cite{McKean56} has shown that the transition semi-group admits a density, $p(t,x,y)$, with respect to $m$ with the following properties:
\begin{enumerate}
\item For each $t>0$ and $(x,y) \in (l, \infty)^2$, $p(t,x,y)=p(t,y,x)>0$.
\item For each $t>0$ and $y \in (l, \infty)$, the maps $x \mapsto p(t,x,y)$ and $x \mapsto A p(t,x,y)$ are continuous and bounded on $(l, \infty)$.
\item $\frac{\partial}{\partial t}p(t,x,y)=A p(t,x,y)$ for each $t>0$ and $(x,y) \in (l, \infty)^2$.
\end{enumerate}
The boundedness of $A p(t,x,y)$ for fixed $t$ and $y \in (l,\infty)$ together with the fact that $m$ has no atom implies via Theorem VII.3.12 in \cite{RY} that for each $y\in (l,\infty)$ the $s$-derivative $\frac{d}{ds}p(t,x,y)$ exists. Since $s$ is differentiable, we have $\frac{d}{dx}p(t,x,y)$ exists. 

Note that if $b\equiv 0$, $s(x)=x$, and the continuity of $\sigma$ and $Ap$ imply, once again by Theorem VII.3.12 in \cite{RY}, that $p(t,x,y)$ is twice continuously differentiable with respect to $x$.

When $b$ is not identically $0$, consider the transformation  $Y_t=s(X_t)$, which yields a one-dimensional  diffusion on $s(\bfE)$ with no drift and inaccessible boundaries. Note that natural boundaries remain so after this transformation. Then, $Y$ possesses a transition density $q$ with respect to its speed measure $\tilde{m}$. Moreover, it can be directly verified that $q(t,x,y)=p(t,s^{-1}(x), s^{-1}(y))$. By the previous discussion $q$ is twice continuously differentiable with respect to $x$. Since $p(t,x,y)=q(t,s(x), s(y)$ and $s$ is twice continuously differentiable, we deduce that $p$ is twice continuously differentiable, as well. This shows that $p(\cdot,\cdot, y) \in C^{1,2}((0,\infty)\times (l,\infty))$ for $y \in (l,\infty)$.

If $l$ is finite but an entrance boundary, then $p(\cdot,\cdot, l) \in C^{1,2}((0,\infty)\times (l,\infty))$ as well. Indeed, Chapman-Kolmogorov identity implies for $s<t$
\[
p(t,y,l)=\int_l^{\infty} p(t-s,y,z)p(s,z,l)m(dz).
\]
Since $p(s, \cdot,\cdot)$ is symmetric $\int_l^{\infty} p(s,z,l)m(dz)=1$. Thus, the assertion follows from differentiating under the integral sign and the analogous properties for $p(t-s,y,z)$. 

Thus, for  $x \in (l,\infty)$ and $z \in (l, \infty)$  (resp. $z \in [l, \infty)$) if $l$ is natural (resp. entrance) boundary, letting $h(t,y)=p(T^*-t,y,z)$, we see that in view of Proposition \ref{smb:l:ubdd} the conditions of Theorem \ref{smb:th:MB} are satisfied, and the Markov bridge from $x$ to $z$ is the unique weak solution of
\be \label{smb:e:1d}
X_t =x +\int_0^t \left\{b(X_u) + a(X_u) \frac{p_x(T^*-u, X_u,z)}{p(T^*-u,X_u,z)}\right\}du+\int_0^t \sigma(X_u)dB_u.
\ee
If $b$ and $\sigma$ are in addition locally Lipschitz, Theorem \ref{smb:th:sunique} ensures the existence and uniqueness of a strong solution of the above SDE.
\end{example}
\begin{remark} In the case of one-dimensional time-homogeneous diffusions on $\bbR$ (\ref{smb:e:iab}) is satisfied under the standard assumption on the drift coefficient having at most a linear growth. To see this, suppose $l=-\infty$, $a$ is strictly positive  and continuous on $\bbR$, and  $b$ satisfies
\[
|b(x)|<K(1+|x|), 
\]
for some $K<\infty$. 

Indeed, first observe that for any $x<y$,  $\int_x^y \frac{|b(z)|}{a(z)}dz <\infty$. Thus, $s$ is well-defined. Moreover, for $x>1$, 
\bean
m(0,x)s'(x)&=& \frac{\int_0^x\frac{2}{a(y)}\exp\left(2\int_0^y \frac{b(z)}{a(z)}dz\right)dy}{\exp\left(2\int_0^x \frac{b(z)}{a(z)}dz\right)}\\
&=&\frac{\int_0^x\frac{2}{a(y)}\exp\left(2\int_0^y \frac{b(z)}{a(z)}dz\right)dy}{1+\int_0^x\frac{2b(y)}{a(y)}\exp\left(2\int_0^y \frac{b(z)}{a(z)}dz\right)dy}\\
&\geq& \frac{\int_0^x\frac{2}{a(y)}\exp\left(2\int_0^y \frac{b(z)}{a(z)}dz\right)dy}{1+\int_0^x\frac{2|b(y)|}{a(y)}\exp\left(2\int_0^y \frac{b(z)}{a(z)}dz\right)dy}\\
&\geq& \frac{\int_0^x\frac{2}{a(y)}\exp\left(2\int_0^y \frac{b(z)}{a(z)}dz\right)dy}{1+K(1+x)\int_0^x\frac{2}{a(y)}\exp\left(2\int_0^y \frac{b(z)}{a(z)}dz\right)dy}\\
&\geq& \frac{1}{K_0 + Kx},
\eean
where 
\[
K_0=K+\frac{1}{\int_0^1\frac{2}{a(y)}\exp\left(2\int_0^y \frac{b(z)}{a(z)}dz\right)dy}.
\]
Since $\int_1^{\infty} \frac{1}{K_0 + Kx}dx=\infty$, we deduce that $\infty$ is an inaccessible boundary. The case of $-\infty$ is handled similarly. 
\end{remark}

\begin{example} Consider the following multi-dimensional linear SDE:
\[
dX_t= \sigma(t)dB_t + \left\{b(t) + \gamma(t) X_t\right\}dt,
\]
where $\sigma(t)$ and $\gamma(t)$ are $d\times d$-matrices, and $b(t)$ is $d$-dimensional vector. Assume that for all $i,j=1, \ldots, d$, $\sigma_{ij}, \gamma_{ij}$ and $b_i$ are continuous on $[0,T^*]$, and the following uniform ellipticity holds:
\[
\sum_{i,j} a_{ij}(t)\lambda_i\lambda_j \geq c \|\lambda\|^2, \qquad\forall x, \lambda \in \bbR^d \mbox{ for some } c>0.\] 

Then, the above SDE has a unique strong solution, which is a Gaussian semimartingale. Moreover, (see Section 10.2 in \cite{kali})
\bean 
m(s,t,x):=E[X_t|X_s=x]=F(t)F^{-1}(s)x + F(t)\int_s^t F^{-1}(u) b(u)du \\
\Sigma(s,t):=E(X_t-m(s,t,X_s))^2|X_s=x]=F(t)\int_s^t \left(F^{-1}(u)\sigma(u)\right)\left(F^{-1}(u)\sigma(u)\right)^*du F^*(t),
\eean
where $F^{-1}$ is the solution of the equation
\[
\frac{dF^{-1}(t)}{dt}=-F^{-1}(t) \gamma(t), \qquad F^{-1}(0)=I.
\]
When $\gamma \equiv 0$, the smoothness of the transition density of $X$ follows from smoothness of the fundamental solution of $u_t = A_t u$ by Theorem 10 of Chap. 1 of \cite{friedman}. Moreover, Assumption \ref{smb:a:pbdd} is satisfied due to the estimates of the fundamental solution given by Theorem 11 of Chap. 1 of \cite{friedman} that have the  form (\ref{smb:e:fud}). Strict positivity of the fundamental solution follows from Theorem 11 of Chap. 2 of \cite{friedman}. Thus, Theorems \ref{smb:th:MB} and \ref{smb:th:sunique} apply to give the existence and uniqueness of a strong solution to the SDE 
\[
X_t=x +\int_0^t \sigma(s)dB_s +\int_0^t \left\{b(s) -a(s)\Sigma^{-1}(s,T^*)\left(z-X_s -\int_s^{T^*}b(u)du\right)\right\}ds.
\]
The general case follows from the transformation $Y_t=F^{-1}(t)X_t $:
\bean
X_t&=& x+ \int_0^t \sigma(s)dB_s +\int_0^t \left\{b(s) + \gamma(s) X_s\right\}ds\\
&&-\int_0^t a(s)(F(T^*)F^{-1}(s))^*\Sigma^{-1}(s,T^*)\left(z-F(T^*)F^{-1}(s)X_s -F(T^*)\int_s^{T^*}F^{-1}(u)b(u)du\right)ds.
\eean
\end{example}

\section{Weak conditioning} 
As mentioned in Section 2 $h$-functions can be employed to obtain weak conditioning. Since $h$-functions lead to positive martingales, one can use them to change the measure. The advantage of such measure changes is that it preserves the Markov property. The proof of this fact will be based on Lemma \ref{smb:l:Mft}, which is a minor modification of Theorem 4.2.1 (ii) in \cite{SV}.

\begin{lemma} \label{smb:l:Mft} Let $f \in C^{1,2}([0,T^*]\times\bfE)$ if $T^*<\infty$; or $f \in C^{1,2}([0,T^*)\times\bfE)$ if $T^*=\infty$. Then, $\left((M^f_t)_{t \in \bfT_s}, (\cB_t)_{t \in \bfT_s}\right)$ is a local martingale for any solution $P^{s,x}$ of the local martingale problem for $A$, where
\[
M^f_t= f(t,X_t)-f(s,X_s)-\int_s^t \left\{\frac{\partial }{\partial u}f(u,X_u) +A_u f(u,X_u) \right\}du.
\]
\end{lemma}

\begin{theorem} \label{sbm:th:hpath} Suppose that the conditions of Assumption \ref{smb:a:A} hold and $T^* <\infty$. Let $h$ be an $h$-function such that $h(T^*,\cdot)>0$ and $h \in C^{1,2}([0,T^*]\times \bfE)$.  Define  $P^{h;s,x}$ on $(\Om, \cB_{T^*})$ by $\frac{dP^{h;s,x}}{dP^{s,x}}=\frac{h(T^*,X_{T^*})}{h(s,x)}$. Then, $P^{h;s,x}$ is the unique solution of the local martingale problem for $A^h$ starting from $x$ at $s$, where
\[
A^h_t =A_t+ \sum_{i.j=1}^d a_{ij}(t,x)\frac{\frac{\partial h}{\partial x_j}(t,x)}{h(t,x)}\frac{\partial}{\partial x_i}.
\]
Consequently, $X$ is a strong Markov process under every $P^{h;s,x}$ for $s<T^*$, $x \in \bfE$   and the associated transition function, $(P^h_{s,t})$ is related to $(P_{s,t})$ via
\be \label{smb:e:htff}
P^h_{s,t}(x,E)=\frac{1}{h(s,x)} \int_E h(t,y)P_{s,t}(x,dy), \qquad x \in \bfE, E\in \sE, t \in \bfT_s.
\ee
\end{theorem}
\begin{proof}
Consider an $f \in C^{\infty}(\bfE)$ and let
\[
M^f_t(h)= f(X_t)-f(X_s) -\int_s^t A^h_r f(X_r)\,dr.
\]
Observe that
\be \label{smb:e:mh-m}
M^f_t(h)-M^f_t=-\sum_{i,j=1}^d \int_s^t  a_{ij}(v,X_v)\frac{\frac{\partial h}{\partial x_j}(v,X_v)}{h(v,X_v)}\frac{\partial f}{\partial x_i}(X_v)\,dv.
\ee
Thus, if we let $\tau_n=T^*\wedge \inf\{t \geq s: |M^f_t(h)| \geq n\} \wedge \inf\{t \geq s: \|X_t\| \geq n\}$, then $\tan_n$ is a stopping time and  $\tau_n \rar T^*$, $P^{h;s,x}$-a.s. as $n \rar \infty$ since $M^f$ and $X$ are  continuous  under $P^{s,x}$, and $P^{s,x} \sim P^{h;s,x}$. In particular,  (\ref{smb:e:mh-m}) also entails $(M^f_{t \wedge \tau_n})$ is a bounded martingale under $P^{s,x}$. We will now see that $(M^f_{t \wedge \tau_n}(h))$ is a martingale under $P^{h;s,x}$. Indeed, for any $u \in [s,t]$ and $E \in \cB_u$, 
\bean
&&h(s,x)E^{h;s,x}\left[\left(M^f_{t\wedge \tau_n}(h)-M^f_{u}(h)\right)\chf_{[\tau_n>u]}\chf_E\right]=E^{s,x}\left[h(t,X_t) \left(M^f_{t\wedge \tau_n}(h)-M^f_{u}(h)\right)\chf_{[\tau_n >u]}\chf_E\right]\\
&=&E^{s,x}\left[h(t,X_t) \left(M^f_{t\wedge \tau_n}-M^f_{u}\right)\chf_{[\tau_n>u]}\chf_E\right]\\
&&-E^{s,x}\left[\chf_{[\tau_n>u]}\chf_E\sum_{i,j=1}^d \int_u^{t\wedge \tau_n} h(t,X_t)  a_{ij}(v,X_v)\frac{\frac{\partial h}{\partial x_j}(v,X_v)}{h(v,X_v)}\frac{\partial f}{\partial x_i}(X_v)\,dv\right]\\
&=&E^{s,x}\left[h(t,X_t) \left(f(X_{t\wedge \tau_n})-f(u,X_u)\right)\chf_{[\tau>u]}\chf_E\right]\\
&&-E^{s,x}\left[\chf_{[\tau_n>u]}\chf_Eh(t,X_t) \int_u^{t \wedge \tau_n} A_v f(X_v)\,dv\right]\\
&&- E^{s,x}\left[\chf_{[\tau_n>u]}\chf_E\sum_{i.j=1}^d \int_u^{t\wedge \tau_n}  a_{ij}(v,X_v)\frac{\partial h}{\partial x_j}(v,X_v)\frac{\partial f}{\partial x_i}(X_v)\,dv\right],
\eean
where the last equality follows from the martingale property of $h(t,X_t)$ under $P^{s,x}$. Letting $g= h f$, observing \[
E^{s,x}\left[h(t,X_t) \left(f(X_{t\wedge \tau_n})-f(u,X_u)\right)\chf_{[\tau_n>u]}\chf_E\right]=E^{s,x}\left[\left(g(t\wedge \tau_n, X_{t\wedge \tau_n})-g(u,X_u)\right)\chf_{[\tau_n>u]}\chf_E\right]
\]
 by the same martingale argument, and utilising Lemma \ref{smb:l:Mft}, we have
\bean
&&h(s,x) E^{h;s,x}\left[\left(M^f_{t\wedge \tau_n}(h)-M^f_{u}(h)\right)\chf_{[\tau_n>u]}\chf_E\right]\\
&=&E^{s,x}\left[\left(M^g_{t\wedge \tau_n}-M^g_u+\int_u^{t \wedge \tau_n} \left\{f(X_v) \frac{\partial h}{\partial u}(v, X_v)+ A_v g(v,X_v)\right\}\,dv\right) \chf_{[\tau_n>u]}\chf_E\right]\\
&&-E^{s,x}\left[ \int_u^{t \wedge \tau_n} h(t,X_t) A_v f(X_v)\,dv\chf_{[\tau_n>u]}\chf_E\right]\\
&&- E^{s,x}\left[\chf_{[\tau_n>u]}\chf_E\sum_{i,j=1}^d \int_u^{t\wedge \tau_n}  a_{ij}(u,X_u)\frac{\partial h}{\partial x_j}(v,X_v)\frac{\partial f}{\partial x_i}(X_v)\,dv\right]\\
&=&E^{s,x}\left[\chf_{[\tau_n>u]}\chf_E \int_u^{t \wedge \tau_n} \left\{f(X_v) \frac{\partial h}{\partial u}(v, X_v)+ A_v g(v,X_v)\right\}dv\right]\\
&&-E^{s,x}\left[\chf_{[\tau_n>u]}\chf_E \int_u^{t \wedge \tau_n} h(v,X_v) A_v f(X_v)\,dv\right]\\
&&-E^{s,x}\left[\chf_{[\tau_n>u]}\chf_E \sum_{i,j=1}^d \int_u^{t\wedge \tau_n}  a_{ij}(u,X_u)\frac{\partial h}{\partial x_j}(v,X_v)\frac{\partial f}{\partial x_i}(X_v)\,dv\right]=0,
\eean
where the second equality holds since $(M^g_{t \wedge \tau_n})$ is a bounded martingale due to the boundedness of $(M^f_{t \wedge \tau_n})$ and the smoothness of $h$, and the last equality follows from the identity $\frac{\partial h}{\partial u}(v, x)+A_v h(v,x)=0$.  As $\tau_n \rar T^*$, $P^{h;s,x}$-a.s., we conclude that $P^{h;s,x}$ solves the local martingale problem. The uniqueness follows easily due to the one-to-one relationship between $P^{h;s,x}$ and $P^{s,x}$ since the local martingale problem for $A$ is well-posed.

The strong Markov property is a direct consequence of the well-posedness of the local martingale problem for $A^h$ via Theorem 4.4.2 in \cite{EK}. The form of the transition function follows directly from the explicit absolute continuity relationship between the measures $P^{h;s,x}$ and $P^{s,x}$.
\end{proof}
The last theorem gives us a conditioning on the path space when $T^*<\infty$ and $h$ satisfies the conditions of the theorem. The coordinate process after this conditioning is often referred to as the {\em h-path process} in the literature. Note that the $h$-function of Example \ref{smb:x:h1} does not satisfy the conditions of the above theorem as $h(T^*, \cdot)$ is not strictly positive. This implies that we cannot use this theorem to condition the coordinate process to end up in a given set. However, $h(t,\cdot)$ is strictly positive and smooth for any $t<T^*$, which allows us to extend the results of the previous theorem  to the case when $h(T^*, \cdot)$ does not satisfy the conditions as well as when $T^*=\infty$. For making this extension possible we first introduce a new canonical space $C([0,T^*), \bfE)$ and   $\cB^{-}_t=\sigma(X_s; s \leq t), \, \cB^{-}_{T^*}=\vee_{t<T^*} \cB^{-}_t$, where $X$ is the coordinate process on $C([0,T^*), \bfE)$. Note that there is no difference between the $\sigma$-algebras $\cB^{-}_t$  and  $\cB_t$  when $T^*=\infty$. The main difference in the case $T^*<\infty$  is the measurable space on which each $\sigma$-algebra is defined. While the former is defined on the space of functions that are continuous on $[0,T^*)$, the latter is defined on the paths that are continuous on $[0,T^*]$. In particular, the functions that are divergent as $t \rar T^*$ belong to the former but not the latter, which in turn implies  $\cB^{-}_{T^*}$ has more elements than  $\cB_{T^*}$. On the other hand, one can easily verify  that there is a one-to-one correspondence between the members of $\cB^{-}_t$  and  those of $\cB_t$ for $t<T^*$. In view of these observations the following fact can be established as a special case of Theorem 1.3.5 in \cite{SV}. 
\begin{theorem} \label{smb:th:Pextend} Let $(t_n)$ be an increasing sequence of deterministic times with $t_n<T^*$ for each $n$ and  suppose  that for each $n$ there exists a probability measure $P^n$ on $(C([0,T^*), \bfE), \cB^{-}_{t_n})$. If $P^{n+1}$ agrees with $P^n$ on $\cB^{-}_{t_n}$ and $\lim_{n \rar \infty}t_n=T^*$, then there exists a unique probability measure, $P$, on $(C([0,T^*), \bfE), \cB^{-}_{T^*})$ that agrees with $P^n$ on $\cB_{t_n}^-$ for all $n \geq 0$.
\end{theorem}

\begin{corollary} \label{smb:c:hpath} Suppose that the conditions of Assumption \ref{smb:a:A} hold and let $h$ be an $h$-function. For any $s <T^*$ and $x \in \bfE$, there exists a unique probability measure $P^{h;s,x}$ on\footnote{Recall that $\Om=C([0,T^*], \bfE)$ (resp. $\Om=C([0,\infty), \bfE)$)  when $T^*<\infty$ (resp. $T^*=\infty$).} $(\Om, \cB_{T^*})$ which solves the local martingale problem for $A^h$, where
\[
A^h_t =A_t+ \sum_{i.j=1}^d a_{ij}(t,x)\frac{\frac{\partial h}{\partial x_j}(t,x)}{h(t,x)}\frac{\partial}{\partial x_i}
\]
on $[0, T] \times \bfE$ starting from $x$ at $s$ for any $T<T^*$. 
Consequently, $X$ is a strong Markov process under every $P^{h;s,x}$ for $s<T^*$, $x \in \bfE$ and the associated transition function, $(P^h_{s,t})$ is related to $(P_{s,t})$ via
\be \label{smb:e:htf}
P^h_{s,t}(x,E)=\frac{1}{h(s,x)} \int_E h(t,y)P_{s,t}(x,dy), \qquad x \in \bfE, \, E\in \sE,  t \in \bfT_s.
\ee
\end{corollary}
\begin{proof} Suppose $T^*=\infty$ and let $T<\infty$. Define a probability measure, $Q^T$, on $(\Om, \cB_{T^*})$ via $\frac{dQ^T}{d P^{s,x}}=\frac{h(T,X_T)}{h(s,x)}$. Then, by Theorem \ref{sbm:th:hpath}, restriction of $Q^T$ to $\cB_T$ is the unique solution to the local martingale problem for $A^h$ starting from $x$ at time $s$ on $[0,T]\times \bfE$. Moreover, $Q^T$ and $Q^{t+T}$ agree on $\cB_T$ for all $t>0$. Indeed, for any $E\in \cB_T$, 
\[
h(s,x) Q^{t+T}(E)=E^{s,x}[h(t+T, X_{t+T})1_E]=E^{s,x}[h(T, X_{T})1_E]=h(s,x)Q^T(E),
\]
implying $Q^{t+T}(E)=Q^{T}(E)$ since $h$ is strictly positive. Thus, by Theorem \ref{smb:th:Pextend}, there exists a unique measure $Q$ on $(C([0,\infty),\bfE), \cB_{\infty})$, solving the local martingale problem for $A^h$ on $[0,T]\times \bfE$ for all $T<\infty$. Consequently,  the local martingale problem is well-posed and  the strong Markov property follows. 

The case $T^*<\infty$ requires more care. Let $\hat{P}^{s,x}$ be the law of $(X_t)_{t \in [s, T^*)}$ under $P^{s,x}$, where $X$ is the coordinate process of $C([0,T^*],\bfE)$.  $\hat{P}^{s,x}$ is a probability measure on $(C([0,T^*),\bfE), \cB_{T^*}^{-})$ such that the corresponding coordinate process admits limits as $t \uar T^*$ with probability $1$. 

Using the function $h$ as a measure change we can again obtain a sequence of measures $(Q^T)_{T<T^*}$ with the property that $Q^T$ and $Q^{t+T}$ agree on $\cB^{-}_T$ for all $0\leq t<T^*-T$. Theorem \ref{smb:th:Pextend} now yields a probability measure, $Q$, on $(C([0,T^*),\bfE), \cB_{T^*}^{-})$ that agrees with $Q^T$ on $\cB^{-}_T$ for $T<T^*$.

We will use this $Q$ to construct the $P^{h;s,x}$ on $(C([0,T^*],\bfE), \cB_{T^*})$. However, in order to do so, we need to establish that under $Q$ the coordinate process, $X$,  of  $C([0,T^*),\bfE)$ admits a limit as $t\uar T^*$.  We will achieve  this by showing  that $Q$ is absolutely continuous with respect to  $\hat{P}^{s,x}$ on $\cB_{T^*}^{-}$. Indeed, for any $E \in \cB_t^{-}$ for some $t<T^*$, we have
\be \label{smb:e:AC}
Q(E)=\hat{E}^{s,x}\left[\frac{h(t, X_t)}{h(s,x)} \chf_E\right]=\hat{E}^{s,x}\left[L_{T^*} \chf_E\right],
\ee
where $0\leq L_{T^*}=\lim_{t \rar T^*}\frac{h(t, X_{t})}{h(s,x)}=\frac{h(T^*, X_{T^*-})}{h(s,x)}$. The existence of this limit and the exchange of expectation and limit are justified since  $(h(t,X_t))_{t \in [0,T^*)}$ is a positive uniformly integrable $\hat{P}^{s,x}$-martingale. Also note that $\hat{E}^{s,x}[L_{T^*}]=1$. 

Define
\[
\lambda=\left\{E\in \cB_{T^*}^-: Q(E)=\hat{E}^{s,x}\left[L_{T^*} \chf_E\right]\right\}.
\]
Clearly, $\lambda$ satisfies the conditions of Dynkin's $\pi-\lambda$ Theorem (see, e.g., Theorem 1.3.2 in \cite{Billingsley}). Moreover, 
\[
\pi=\{E: E \in \cB_t^{-} \mbox{ for some } t<T^*\} \subset \lambda
\]
is closed under intersection. Thus, by Dynkin's $\pi-\lambda$ Theorem the equality (\ref{smb:e:AC}) holds for all $E \in \cB_{T^*}^{-}$ implying the claimed absolute continuity. 

We now claim that the  sequence $(X_{t_n})$ with $t_n \uar T^*$ is $Q$-a.s. Cauchy.  Observe that on the set 
\[
E=\cap_{m \geq 1}\cup_{n \geq 1} \cap_{k \geq n}\left \{\om: \|X_{t_k}(\om)-X_{t_{k-1}}(\om)\| < \frac{1}{m}\right\} \in \cB_{T^*}^{-},
\]
the sequence is Cauchy.  Moreover, $\hat{P}^{s,x}(E)=1$ which together with (\ref{smb:e:AC}) implies $Q(E)=1$.  Thus, $\lim_{t \rar T^*} X_t$ exists, $Q$-a.s.. 

Next, define $X^h$ by $X^h_t=X_t$ for $t <T^*$, and $X^h_{T^*}=\lim_{t \rar T^*} X_t$. If we denote by $P^{h;s,x}$ the law of $X^h$, then     it is easily seen that it is a probability measure on $(C([0,T^*],\bfE), \cB_{T^*})$. Thus, we have shown in view of (\ref{smb:e:AC}) that for any $E \in \cB_t$
\be \label{smb:e:AC2}
P^{h;s,x}(E) =E^{s,x}\left[\frac{h(t,X_t)}{h(s,x)}\chf_E\right]
\ee
for all $s\leq t\leq T^*$.

To show the uniqueness assume that there exists another measure $\tilde{P}^{h;s,x}$ on $(C([0,T^*],\bfE), \cB_{T^*})$ which solves the local martingale problem for $A^h$. Then, the restriction of this measure to $\cB_T$ for $T<T^*$ solves the local martingale problem for $A^h$ when solutions have sample paths in $C([0,T],\bfE)$. However, this local martingale problem is well-posed due to the one-to-one correspondence with the martingale problem for $A$ when solutions have sample paths in $C([0,T],\bfE)$ via Girsanov transform since $h(T,\cdot)>0$. Thus, by Theorem \ref{smb:th:Pextend} $\tilde{P}^{h;s,x}$ and $P^{h;s,x}$ agree on $(C([0,T^*],\bfE), \cB_{T^*})$. This proves the well-posedness of the local martingale problem for $A^h$ and, therefore, via Theorem 4.4.2 in \cite{EK}, the strong Markov property holds for $X$ under $P^{h;s,x}$.

Finally, the representation for the transition function when $T^*<\infty$ follows from (\ref{smb:e:AC2}). When $T^*=\infty$ the required representation can be obtained from   the one  associated with $Q^T$, which is given by Theorem \ref{sbm:th:hpath}, for any $T$ satisfying $s<t<T$ since $Q$ agrees with $Q^T$ on $\cB_T$.
\end{proof}
In view of the relationship between the solutions of the local martingale problem and the weak solutions of SDEs (Theorem \ref{g:t:lm2ws}) yields Theorem \ref{c:hsde}.

\section{Strong conditioning}
Now we consider the problem of strong conditioning  of a Markov process. Intuitively, one can see that such conditioning can be done via ``$h(t,x)=P^{t,x}(X_T^*=z)$''. As $[X_T^*=z]$ is most likely a null set, the definition in the quotation marks shouldn't be taken too literally.  However, it guides us how to proceed. Suppose that $P_{s,t}$ admits a density $p_{s,t}$ belonging to $C^{1,2}([0,T^*)\times \bfE)$. Thus, if it is strictly positive it can be used as an $h$-function. The problem is that this function explodes at $t=T^*$, which is in fact the very reason why this conditioning works, so we cannot directly use Corollary \ref{smb:c:hpath}. However, it can be applied locally, i.e. until times away from $T^*$, to produce a family of measures on the canonical space. If,  additionally, the family of solutions to the local martingale problem is weakly continuous,  it is possible to demonstrate that this family converges weakly to a probability measure on the canonical space yielding the bridge condition.  In addition we will obtain a stochastic differential equation associated with the bridge process. 
\\

{\bf \sc Proof of Theorem \ref{smb:th:MB}.}\hspace{3mm}
Let $h(t,x)=p(T^*-t,x,z)$ and  define $Q^T$, on $(C([0,T^*],\bfE), \cB_{T^*})$ by $\frac{dQ^T}{dP^{x}}=\frac{h(T,X_T)}{h(0,x)}$. First, we will show that  $(Q^T)$ converge weakly, as $T \rar T^*$,  to a probability measure, $P^{x \rar z}_{0 \rar T^*}$, on $(C([0,T^*],\bfE), \cB_{T^*})$ such that $P^{x \rar z}_{0 \rar T^*}(X_T^*=z)=1$. Usually this is achieved in two steps: 1) Verifying that the family of measures is tight. 2) Demonstrating that the finite-dimensional distributions of the coordinate process under $Q^T$ converge to those under $P^{x \rar z}_{0 \rar T^*}$.  

In view of  Theorem 2.7.2 in \cite{BillingsleyC}  since  $Q^T(X_0=x)=1$ for all $T \in [0,T^*)$,  the tightness will follow once we show that for any  $c>0$
\be \label{smb:e:tight}
\lim_{\delta \rar 0} \limsup_{T \rar T^*} Q^T(w(X,\delta, [0,T^*]) > 8 c)=0,
\ee
where
\[
w(X, \delta, [S,T])= \sup_{\substack{ |s-t|\leq \delta \\ s,t  \in [S,T]} }\|X_s-X_t\|.
\]
We will first obtain some estimates on the modulus of continuity in a neighbourhood of $T^*$.  To this end let $Z_{\delta}=w(X,\delta, [0,\delta])$ and observe that 
\be \label{smb:e:setd}
[Z_{\delta} \circ \theta_{T^*-\delta} > 4c] \subset [ Z_{T^*-T} \circ \theta_{T} > 2 c] \cup [Z_{T-T^*+ \delta} \circ \theta_{T^*-\delta} > 2 c], \qquad \forall T>T^*-\delta.
\ee 
To get an estimate on the probability of the left hand side of the above, we will first consider the first set of the right hand side.
\bea
Q^T( Z_{T^*-T} \circ \theta_{T} > 2c)&=& E^{x}\left[\chf_{[ Z_{T^*-T} \circ \theta_{T} >2 c]}\frac{p(T^*-T, X_T,z)}{p(T^*,x,z)}\right] \nn \\
&=&E^{x}\left[P^{X_T}(Z_{T^*-T}>2 c)\frac{p(T^*-T, X_T,z)}{p(T^*,x,z)}\chf_{[X_T \in B_1(z)]}\right] \nn\\
&&+E^{x}\left[P^{X_T}(Z_{T^*-T}>2 c)\frac{p(T^*-T, X_T,z)}{p(T^*,x,z)}\chf_{[X_T \notin B_1(z)]}\right], \label{smb:e:2est1.2}
\eea
where the first equality is due to the definition of $Q^T$ and the second is the Markov property.

Since $(p(T^*-t, X_t, z))_{t \in [0,T]}$ is a martingale, we have
\be  \label{smb:e:est1.1a}
E^{x}\left[P^{X_T}(Z_{T^*-T}>2 c)\frac{p(T^*-T, X_T,z)}{p(T^*,x,z)}\chf_{[X_T \in B_1(z)]}\right]\leq  \sup_{y \in B_1(z)}P^y(Z_{T^*-T}>2 c).
\ee
Observe that $\lim_{h \rar 0} P^y (Z_h > 2c)=0$. To see this note that the sets $[Z_h >2c]$ are decreasing to a set in $\wtilde{\cF_0}$ and therefore by Blumenthal's zero-one law,  probability of the limiting set is either $0$ or $1$. If the limiting probability is $1$, this implies that $P^y(Z_h >2c)=1$ for all $h$, which in turn means that in every neighbourhood of $0$ there exist a time at which the value of the process is $c$ away from its value at the origin. This contradicts the continuity of $X$, therefore $\lim_{h \rar 0} P^y(Z_h >2c)=0$. 

This observation allows us to conclude that for any compact subset, $K$, of $\bfE$
\be \label{smb:e:est1.1.0}
 \lim_{h \rar 0} \sup_{y \in K}P^y(Z_{h}>2 c)=0.
 \ee
 Indeed, if the above fails, there exists a sequence of $(y_n)$ and $(h_n)$ such  that $h_n \rar 0$, $y_n \rar y \in K$ with
 \[
 0 < \liminf_{n \rar \infty}P^{y_n}(Z_{h_n} >2c) \leq \lim_{m \rar \infty} \liminf_{n \rar \infty}P^{y_n}(Z_{h_m} >2c) = \lim_{m \rar \infty}P^{y}(Z_{h_m} >2c)=0
 \]
by the weak continuity of the laws $P^x$, which is a contradiction. This, together with (\ref{smb:e:est1.1a}), implies that 
 \be  \label{smb:e:est1.1}
\lim_{T \rar T^*} E^{x}\left[P^{X_T}(Z_{T^*-T}>2 c)\frac{p(T^*-T, X_T,z)}{p(T^*,x,z)}\chf_{[X_T \in B_1(z)]}\right] =0.
\ee
The same limit holds for the second term of (\ref{smb:e:2est1.2}). Indeed,
\bean
&&E^{x}\left[P^{X_T}(Z_{T^*-T}>2 c)\frac{p(T^*-T, X_T,z)}{p(T^*,x,z)}\chf_{[X_T \notin B_1(z)]}\right]\leq  E^{x}\left[\frac{p(T^*-T, X_T,z)}{p(T^*,x,z)}\chf_{[X_T \notin B_1(z)]}\right]\\
&=&\frac{1}{p(T^*,x,z)}\int_{B_1^c(z)} p(T^*-T,y,z)p(T,x,y)m(dy),
\eean
which converges to $0$ as $T \rar T^*$ by (\ref{smb:e:dual}) or (\ref{smb:e:dualinf}). 

Combining the above with (\ref{smb:e:est1.1}) and (\ref{smb:e:2est1.2}) yields
\be \label{smb:e:est1}
\lim_{T \rar T^*} Q^T( Z_{T^*-T} \circ \theta_{T} > 2c)=0.
\ee

Next, we will show that $\lim_{\delta \rar 0}\limsup_{T \rar T^*} Q^T(Z_{T-T^*+ \delta} \circ \theta_{T^*-\delta} > 2 c)=0$. Let 
\bean
\tau^{\delta}&:=&\inf\{t\geq 0:\sup_{0 \leq s \leq t}X_s-\inf_{0 \leq s \leq t}X_s>2c\}\wedge \delta\wedge T^*,\\
\tau_c&=&\inf\{t\geq 0: X_t \notin B_{\frac{c}{2}}(X_{0})\}\wedge \delta\wedge T^*,
\eean
where $\inf\emptyset=\infty$.

Observe that 
\be \label{smb:e:tauc}
[Z_{T-T^*+ \delta} \circ \theta_{T^*-\delta} > 2 c]=[T^*-\delta + \tau^{\delta}\circ \theta_{T^*-\delta} <T] \subset [T^*-\delta + \tau_c\circ \theta_{T^*-\delta} <T]. 
\ee
Thus,
\bea
\lim_{T \rar T^*}Q^T(Z_{T-T^*+ \delta} \circ \theta_{T^*-\delta} > 2 c)&=&\lim_{T\rar T^*}\frac{E^x[\chf_{[T^*-\delta + \tau^{\delta}\circ \theta_{T^*-\delta} <T]}p(T^*-T, X_T,z)]}{p(T^*,x,z)} \nn \\
&=&\lim_{T\rar T^*}\frac{E^x[\chf_{[T^*-\delta + \tau^{\delta}\circ \theta_{T^*-\delta} <T]}p(\delta - \tau^{\delta}\circ \theta_{T^*-\delta}, X_{T^*-\delta + \tau^{\delta}\circ \theta_{T^*-\delta} },z)]}{p(T^*,x,z)}\nn \\
&=& \frac{E^x[\chf_{[ \tau^{\delta}\circ \theta_{T^*-\delta} <\delta]}p(\delta - \tau^{\delta}\circ \theta_{T^*-\delta}, X_{T^*-\delta + \tau^{\delta}\circ \theta_{T^*-\delta} },z)]}{p(T^*,x,z)}, \label{smb:e:est2.1T}
\eea
where the second equality follows from the Optional Stopping Theorem applied to the martingale $M:=(p(T^*-t, X_t, z))_{t \in [0,T^*)}$, and the last is due to the Monotone Convergence Theorem. 

\noindent {\bf Case 1:} Suppose Assumption \ref{smb:a:pbdd} holds. Note that the numerator in (\ref{smb:e:est2.1T}) can be rewritten as
\bean
&&E^x[\chf_{[X_{T^*-\delta}\in B_{\frac{c}{4}}(z)]}\chf_{[ \tau^{\delta}\circ \theta_{T^*-\delta} <\delta]}M_{T^*-\delta + \tau^{\delta}\circ \theta_{T^*-\delta}}]+E^x[\chf_{[X_{T^*-\delta}\notin B_{\frac{c}{4}}(z)]}\chf_{[ \tau^{\delta}\circ \theta_{T^*-\delta} <\delta]}M_{T^*-\delta + \tau^{\delta}\circ \theta_{T^*-\delta}}]\\
&&\leq  E^x[\chf_{[X_{T^*-\delta}\in B_{\frac{c}{4}}(z)]}\chf_{[ \tau_c\circ \theta_{T^*-\delta} <\delta]}M_{T^*-\delta + \tau_c\circ \theta_{T^*-\delta}}]+E^x[\chf_{[X_{T^*-\delta}\notin B_{\frac{c}{4}}(z)]}M_{T^*-\delta}]
\eean
Observe that under (\ref{smb:e:tdbound}) we have $\chf_{[X_{T^*-\delta}\in B_{\frac{c}{4}}(z)]}\chf_{[ \tau_c\circ \theta_{T^*-\delta} <\delta]}M_{T^*-\delta + \tau_c\circ \theta_{T^*-\delta}}$ and $\chf_{[X_{T^*-\delta}\notin B_{\frac{c}{4}}(z)]}M_{T^*-\delta}$ uniformly bounded in $\delta$. Thus, dominated convergence theorem yields
\[
\lim_{\delta \rar 0} \lim_{T \rar T^*}Q^T(Z_{T-T^*+ \delta} \circ \theta_{T^*-\delta} > 2 c)=0
\]
in view of $P^x(\lim_{t \rar T^*}M_t=0)=1$. 
Combining the above with (\ref{smb:e:setd}) and (\ref{smb:e:est1}) yields 
\be \label{smb:e:est2.1}
\lim_{\delta \rar 0} \lim_{T \rar T^*} Q^T(Z_{\delta} \circ\theta_{T^*-\delta} >4 c) =0.
\ee

\noindent {\bf Case 2:}  Suppose Assumption \ref{smb:a:potden} holds.
For all $t >0$ and $\delta \leq t$ define 
\[
\varphi(t, \delta,x)=E^x[\chf_{[ \tau^{\delta}\circ \theta_{t-\delta} <\delta ]}p(\delta- \tau^{\delta}\circ \theta_{t-\delta} , X_{t-\delta + \tau^{\delta}\circ \theta_{t-\delta}},z)].
\]
Observe that for every $t>0$ the map $\delta \mapsto \varphi(t,\delta,x)$ is increasing. Indeed, consider $M^t_s=p(t-s, X_s, z)$ and note that in view of Lemma \ref{smb:l:potential} 
\[
\varphi(t,\delta,x)= E^x[M^t_{t-\delta + \tau^{\delta}\circ \theta_{t-\delta}}].
\]
The claim follows since the stopping times $t-\delta + \tau^{\delta}\circ \theta_{t-\delta}$ are decreasing in $\delta$  and $M^t$ is a supermartingale on $[0,t]$.

Due to (\ref{smb:e:est2.1T}) we have 
\be \label{smb:e:phi}
\lim_{T \rar T^*}Q^T(Z_{T-T^*+ \delta} \circ \theta_{T^*-\delta} > 2 c)=\frac{\varphi(T^*,\delta,x)}{p(T^*,x,z)}.
\ee
Our next goal is to show that $\lim_{\delta \rar 0} \varphi(T^*,\delta,x)=0$. The first step towards this goal is to obtain that $\varphi(t,0,x):=\lim_{\delta \rar 0} \varphi(t,\delta,x)=0$ for almost every $t$. Since $\varphi(t, \cdot, x)$ is increasing for every $t>0$, in view of Lemma \ref{smb:l:LT} it is enough to show that 
\[
\lim_{\alpha \rar \infty} \alpha\int_0^{\infty} \int_0^t e^{-\alpha s -\beta t} \varphi(t,s,x) dsdt =0, \; \forall \beta >0.
\]
We will find an upper bound to this Laplace transform by considering the following decomposition.
\bea
\varphi(t, \delta,x)&=&\int_{y \in B_{\frac{c}{4}}(z)} E^y[\chf_{[\tau^{\delta}<\delta]}p(\delta-\tau^{\delta}, X_{\tau^{\delta}},z)]p(t-\delta, x,y)m(dy) \nn \\
&&+\int_{y \in B^c_{\frac{c}{4}}(z)} E^y[\chf_{[\tau^{\delta}<\delta]}p(\delta-\tau^{\delta}, X_{\tau^{\delta}},z)]p(t-\delta, x,y)m(dy) \nn\\
&\leq & \int_{y \in B_{\frac{c}{4}}(z)} E^y[\chf_{[\tau_c<\delta]}p(\delta-\tau_c, X_{\tau_c},z)]p(t-\delta, x,y)m(dy) \label{smb:e:est3.1} \\
&&+ \int_{y \in B^c_{\frac{c}{4}}(z)} p(\delta,y,z) p(t-\delta, x,y)m(dy) \label{smb:e:est3.2},
\eea
where the inequality is due to Lemma \ref{smb:l:potential}. 

Note that
\bean
&&\int_0^{\infty} \int_0^t e^{-\alpha \delta -\beta t} E^y[\chf_{[\tau_c<\delta]}p(\delta-\tau_c, X_{\tau_c},z)]p(t-\delta, x,y) d\delta dt\\
&&= E^y \int_{\tau_c}^{\infty} \int_{\tau_c}^t e^{-\alpha \delta -\beta t} p(\delta-\tau_c, X_{\tau_c},z)p(t-\delta, x,y) d\delta dt \\
&&= E^y \left\{e^{-(\alpha+\beta)\tau_c} \int_{0}^{\infty} \int_{0}^t e^{-\alpha \delta -\beta t} p(\delta, X_{\tau_c},z)p(t-\delta, x,y) d\delta dt \right\}\\
&&= E^y \left\{e^{-(\alpha+\beta)\tau_c} \int_{0}^{\infty} \int_{\delta}^\infty e^{-\alpha \delta -\beta t} p(\delta, X_{\tau_c},z)p(t-\delta, x,y) dt d\delta\right\}  \\
&&=E^y \left\{e^{-(\alpha+\beta)\tau_c} \int_{0}^{\infty} e^{-(\alpha+\beta) \delta }p(\delta, X_{\tau_c},z)d\delta \right\} \int_{0}^\infty e^{-\beta t} p(t, x,y) dt\\
&&=E^y\left[ e^{-(\alpha+\beta)\tau_c}u^{\alpha+\beta}(X_{\tau_c},z)\right]u^{\beta}(x,y).
\eean
Due to (\ref{smb:e:bddpotential}) $\sup_{\alpha >0, w \in \partial B_{\frac{c}{2}(z)}}\alpha u^{\alpha+\beta}(w,z)<\infty$. Moreover, 
\[
\int_{B_{\frac{c}{4}}(z)}u^{\beta}(x,y)m(dy)=\int_0^{\infty} e^{-\beta t}P^x(X_t \in B_{\frac{c}{4}}(z))dt <\infty.
\]
Thus,   the Dominated Convergence Theorem yields
\bea
&&\lim_{\alpha\rar \infty} \alpha \int_0^{\infty} \int_0^t e^{-\alpha \delta -\beta t} \int_{y \in B_{\frac{c}{4}(z)}}E^y[\chf_{[\tau_c<\delta]}p(\delta-\tau_c, X_{\tau_c},z)]p(t-\delta, x,y) m(dy) d\delta dt \nn \\
&&=\int_{y \in B_{\frac{c}{4}(z)}} E^y\left[\lim_{\alpha\rar \infty}  \alpha  e^{-(\alpha+\beta)\tau_c}u^{\alpha+\beta}(X_{\tau_c},z)\right]u^{\beta}(x,y)m(dy) =0 \label{smb:e:est3.3},
\eea
since $P^y(\tau_c=0)=0$ by the continuity of $X$.   

Next we turn to (\ref{smb:e:est3.2}). Note that 
 \[
\int_0^{t} \alpha e^{-\alpha \delta} \int_{y \in B^c_{\frac{c}{4}}(z)} p(\delta,y,z) p(t-\delta, x,y)m(dy) d\delta \leq p(t, x,z)\int_0^{t} \alpha e^{-\alpha \delta} d\delta\leq p(t, x,z).
 \]
Since 
\[
\int_0^{\infty}e^{-\beta t} p(t, x,z)dt=u^{\beta}(x,z)<\infty,
\]
the Dominated Convergence Theorem yields
\bean
&&\lim_{\alpha\rar \infty} \alpha \int_0^{\infty} \int_0^t e^{-\alpha \delta -\beta t} \int_{y \in B^c_{\frac{c}{4}}(z)} p(\delta,y,z) p(t-\delta, x,y)m(dy)  d\delta dt\\
&&= \int_0^{\infty} e^{-\beta t}\lim_{\alpha\rar \infty} \alpha \int_0^t e^{-\alpha \delta}  \int_{y \in B^c_{\frac{c}{4}}(z)} p(\delta,y,z) p(t-\delta, x,y)m(dy)  d\delta dt.
\eean
Since $\int_{y \in B^c_{\frac{c}{4}}(z)} p(\delta,y,z) p(t-\delta, x,y)m(dy)\leq p(t,x,z)$ and converges to $0$ as $\delta \rar 0$ in view of (\ref{smb:e:dualinf}), Lemma \ref{smb:l:LT} yields  
\[
\lim_{\alpha\rar \infty} \alpha \int_0^{\infty} \int_0^t e^{-\alpha \delta -\beta t} \int_{y \in B^c_{\frac{c}{4}}(z)} p(\delta,y,z) p(t-\delta, x,y)m(dy)  d\delta dt=0.
\]
Combining the above with (\ref{smb:e:est3.3}), (\ref{smb:e:est3.2}), and (\ref{smb:e:est3.1}) gives 
\[
\lim_{\alpha \rar \infty} \alpha\int_0^{\infty} \int_0^t e^{-\alpha s -\beta t} \varphi(t,s,x) dsdt =0, \; \forall \beta >0.
\]
Thus, $\varphi(t,0,x)=0$ for almost every $t>0$.

Define $D:=\{t>0:\varphi(t,0,x)=0\}$, let  $S \in D$ and consider $s <S$. Then, by Chapman-Kolmogorov identity we have
\bean
0=\varphi(S,0,x)&=&\lim_{\delta \rar 0} \varphi(S,\delta,x) \\
&=&   \lim_{\delta \rar 0} \int \varphi(s,\delta,y) p(S-s,x,y)m(dy)\\
&=&\int \varphi(s,0,y) p(S-s,x,y)m(dy),
\eean
where the last equality is due to the Dominated Convergence Theorem since $\varphi(s,\delta,y)\leq p(s,y,z)$.  Thus, $\varphi(s,0,y) p(S-s,x,y)=0$ for $m$-a.e. $y$, which implies $\varphi(s,0,y)=0, \, m$-a.e. if $p(t,x,y)>0$ for all $t>0$ and $m$-a.e. $y$. On the other hand, if $t \mapsto p(t,x,y)$ is continuous on $(0, \infty)$,  the fact that $D$ is dense in $(0, \infty)$ yields that for any $t>s$, $\varphi(s,0,y) p(t-s,x,y)=0$ for $m$-a.e. $y$. Therefore, under either assumption, we have 
\bean
\varphi(T^*,0,x)&=&\lim_{\delta \rar 0} \int \varphi(s,\delta,y) p(T^*-s,x,y)m(dy) \\
&=&\int \varphi(s,0,y) p(T^*-s,x,y)m(dy)=0.
\eean
Combining the above with (\ref{smb:e:setd}) and (\ref{smb:e:est1}) yields 
\be \label{smb:e:est2.2}
\lim_{\delta \rar 0} \lim_{T \rar T^*} Q^T(Z_{\delta} \circ\theta_{T^*-\delta} >4 c) =0
\ee
in view of (\ref{smb:e:phi}).

The above as well as (\ref{smb:e:est2.1}) implies that for any $\eps>0$ there exist $\hat{\delta}$ and $\hat{T} > T^*-\hat{\delta}$ such that for all $T>\hat{T}$ we have
\[
Q^T(w(X, \hat{\delta}, [T^*-\hat{\delta}, T^*])> 4 c)< \frac{\eps}{2}.
\]
As $w(X, \delta, [u,v])$ is increasing in $\delta$, we can conclude
\[
Q^T(w(X, \delta, [T^*-\hat{\delta}, T^*])> 4 c)< \frac{\eps}{2}, \qquad \forall \delta <\hat{\delta}, T>\hat{T}.
\]
Since
\[
Q^T(w(X,\delta, [0,T^*]) >8 c) \leq Q^T(w(X,\delta, [0,T^*-\hat{\delta}]) >4 c) + Q^T(w(X,\delta, [T^*-\hat{\delta},T^*]) >4 c),
\]
it remains to show that
\[
\lim_{\delta \rar 0} \limsup_{T \rar T^*} Q^T(w(X,\delta, [0,T^*-\hat{\delta}]) >4 c) =0.
\]
However, for $T> \hat{T}$, one has $T^*-\hat{\delta}<T$, thus for such $T$
\[
Q^T(w(X,\delta, [0,T^*-\hat{\delta}]) >4 c) =E^x\left[ \chf_{[w(X,\delta, [0,T^*-\hat{\delta}]) >4 c]}\frac{p(T^*-\hat{T}, X_{\hat{T}},z)}{p(T^*,x,z)}\right].
\]
Then, by the Dominated Convergence Theorem,
\[
\lim_{\delta \rar 0} \limsup_{T \rar T^*} Q^T(w(X,\delta, [0,T^*-\hat{\delta}]) >4 c) =E^x\left[ \lim_{\delta \rar 0} \chf_{[w(X,\delta, [0,T^*-\hat{\delta}]) >4 c]}\frac{p(T^*-\hat{T}, X_{\hat{T}},z)}{p(T^*,x,z)}\right]=0.
\]
To show the convergence of the finite dimensional distributions consider $[0, T^*)$ as a dense subset of $[0,T^*]$ and note that for any finite set ${t_1, \ldots, t_k} \subset [0,T^*)$ and a bounded continuous function $f:\bfE^k \mapsto \bbR$
\[
\lim_{T^ \rar T^*} E^{Q^T}[f(X_{t_1}, \ldots, X_{t_k})]=E^{Q^{t_k}}[(X_{t_1}, \ldots, X_{t_k})],
\]
which establishes the desired convergence.

Thus, the sequence of measures, $(Q^T)$, has a unique limit point $P^{x \rar z}_{0 \rar T^*}$ on $(C([0,T^*],\bfE), \cB_{T^*})$. Moreover, its restriction, $Q^T$, to  $(C([0,T],\bfE), \cB_{T})$ for any $T<T^*$ is a solution to the local martingale problem for $A^p$ on  $[0,T] \times \bfE$.  In particular, for any $E \in \cB_T$, we have $P^{x \rar z}_{0 \rar T^*}(E)=E^{x}\left[\frac{h(T,X_T)}{h(0,x)}\chf_E\right]$.

Next, we show the bridge condition. To this end pick $f \in C^{\infty}_K(\bfE)$,  $\eps >0$ and consider $r>0$ such that $\sup_{y \in B_r(z)}|f(y)-f(z)|<\eps$. Thus,
\bean
 &&E^{x \rar z}_{0 \rar T^*}[f(X_{T^*})]=\lim_{T \rar T^*} E^{x \rar z}_{0 \rar T^*}[f(X_{T})]=\lim_{T \rar T^*}E^{x}\left[\frac{h(T,X_T)}{h(0,x)}f(X_{T})\right]\\
 &=&f(z)+\lim_{T \rar T^*}\frac{ E^{x}\left[p(T^*-T,X_T,z)(f(X_{T})-f(z))\right]}{p(T^*,x,z)}\\
 &=&f(z)+\lim_{T \rar T^*}\int_{B_r(z)} \frac{p(T^*-T,y,z)p(T,x,y)}{p(T^*,x,z)}\left(f(y)-f(z)\right)m(dy) \\
 &&+ \lim_{T \rar T^*}\int_{B_r^c(z)} \frac{p(T^*-T,y,z)p(T,x,y)}{p(T^*,x,z)}\left(f(y)-f(z)\right)m(dy).
 \eean
Since $f$ is bounded the second integral above converges to $0$ in view of (\ref{smb:e:dual}) or (\ref{smb:e:dualinf}). 
 Moreover,
\[
\left|\int_{B_r(z)} \frac{p(T^*-T,y,z)p(T,x,y)}{p(T^*,x,z)}\left(f(y)-f(z)\right)m(dy) \right|\leq \eps,
\]
which implies the bridge condition by the arbitrariness of $\eps$.

Existence of a weak solution to (\ref{smb:e:sdeB}) follows from Girsanov's Theorem. Indeed, since the martingale problem for $A$ is well-posed there exists a unique weak solution to 
\[
X_t=x+\int_0^t b(X_u)du+\int_0^t \sigma(X_u)\,dB_u,
\]
by Theorem \ref{g:t:lm2ws}. With an abuse of notation, we denote the associated probability measure with $P^x$. The above considerations show that there exists a probability measure, $P^{x \rar z}_{0 \rar T^*}$, which is locally absolutely continuous with respect $P^x$ in the sense that for any $t<T^*$ and $E\in \cF_t$ $P^{x \rar z}_{0 \rar T^*}(E)=E^x\left[\chf_E \frac{h(t,X_t)}{h(0,x)}\right]$. Thus, an application of Girsanov's Theorem yields the conclusion. 

To show that $\inf_{t \in [0,T]} h(t,X_t)>0$, $P^{x \rar z}_{0 \rar T^*}$-a.s. for any $T<T^*$ consider $T_n=\inf\{t\geq 0: h(t,X_t)\leq \frac{1}{n}\}$ and observe that  $P^{x \rar z}_{0 \rar T^*}(T_n <t)=E^x\left[\chf_{[T_n<t]}\frac{h(t,X_t)}{h(0,x)}\right]\leq \frac{1}{n h(0,x)} \rar 0$ as $n \rar \infty$. 

To show the Markov property let $s<t<T^*$, $f$ be bounded and measurable, and $E \in \cB_s$. Then,
\[
E^{x \rar z}_{0 \rar T^*}[f(X_t)\chf_E]=E^x[f(X_t)h(t,X_t)\chf_E]=E^{x \rar z}_{0 \rar T^*}\left[\frac{E^x[f(X_t)h(t,X_t)|X_s]}{h(s,X_s)}\chf_E\right],
\]
where we used the positivity of $h(t,X_t)$ under $P^x$ for the second equality. 

Similarly, if $g$ is also bounded and measurable, we have in addition
\[
E^{x \rar z}_{0 \rar T^*}[f(X_t)g(X_s)]=E^{x \rar z}_{0 \rar T^*}\left[\frac{E^x[f(X_t)h(t,X_t)|X_s]}{h(s,X_s)}g(X_s)\right].
\]
This shows the Markov property and yields the representation of the transition function.

Finally, the weak uniqueness can be proved following the same steps as in the proof of Corollary \ref{smb:c:hpath}.
\qed
\\

Clearly, if the conditions of Theorem \ref{smb:th:MB} are satisfied by all $x \in \bfE$ and $h(t, \cdot)>0$ for $t<T^*$, there exists a unique weak solution to  (\ref{smb:e:smSDE}) for any $s<T^*$ and $x \in \bfE$. Thus, Theorem \ref{g:t:lm2ws} implies  well-posedness of the local martingale problem which in turn yields the strong Markov property of its solutions. More precisely, the following holds.
\begin{corollary} \label{smb:c:smarkov} Let $\sigma$ be a matrix field such that $\sigma \sigma^*=a$.  Suppose that Assumption \ref{smb:a:pbdd} or \ref{smb:a:potden} is in force, $T^*<\infty$. Fix $z \in \bfE$ satisfying $m(\{z\})=0$ and define $h(t,y):=p(T^*-t,y,z)$ such that the following conditions hold for all $x \in \bfE$:
\begin{enumerate}
\item $h(t,x)>0$ for all $t \in [0,T^*)$.  
\item $h \in C^{1,2}([0,T^*), \bfE)$.
\item If Assumption \ref{smb:a:potden} is enforced, then
\begin{itemize}
\item[i)] $u^{\alpha}(x,z)<\infty$ for $\alpha>0$,
\item[ii)] either the map $t\mapsto p(t,x,y)$ is continuous on $(0, \infty)$ for every $y \in \bfE$, or for all $t>0$ $p(t,x,y)>0, \,m$-a.e. $y$. 
\end{itemize}
\end{enumerate}
Then  there exists a unique weak solution  on $[s,T^*]$  to 
\be \label{smb:e:smSDE}
X_t=x+\int_s^t \left\{ b(X_u)+ a(X_u)\frac{(\nabla h(u,X_u))^*}{h(u,X_u)}\right\}du+\int_s^t \sigma(X_u)\,dB_u,
\ee
the law of which, $P^{x \rar z}_{s \rar T^*}$, satisfies  $P^{x \rar z}_{s \rar T^*} (\inf_{u \in [s,T]}h(u,X_u) =0)=0$ for any $T<T^*$, and $P^{x \rar z}_{s \rar T^*}(X_{T^*}=z)=1$. 
 Moreover, the solution has the strong Markov property.
\end{corollary}

To obtain strong solutions to (\ref{smb:e:sdeB})  we need to impose stronger conditions on the transition density $p$ and the coefficients $a$ and $b$. These conditions will imply the pathwise uniqueness which in turn will lead to the existence of a strong solution to (\ref{smb:e:sdeB}) in view of Yamada-Watanabe Theorem. The strict positivity of $p$ in $\mbox{int}(\bfE)$, which we will require, is not too restrictive and it is always satisfied in the case of one-dimensional diffusions (see \cite{McKean56}). On the other hand, the condition $p(t,y,y')>0$ for all $t>0, y \in \mbox{int}(\bfE)$ and $y' \in \partial(\bfE)$ is more delicate and its fulfilment depends on the classification of the boundaries of the underlying diffusion. 
\begin{theorem} \label{smb:th:sunique}  Let $\sigma$ be a matrix field such that $\sigma \sigma^*=a$ and fix $x \in \mbox{int}(\bfE)$ and $z \in \bfE$ such that the hypotheses of Theorem \ref{smb:th:MB} hold. Suppose, in addition,   that for any closed set $C \subset \mbox{int}(\bfE)$ 
\[
\|b(y)-b(y')\|+\|\sigma(y)-\sigma(y')\|\leq K_C \|y-y'\|,  \, y, y' \in C,
\]
as well as  $p(t,y,z)>0$ for all $t \in (0,T^*]$ and $y \in \mbox{int}(\bfE)$. If  $P^x(\inf\{t>0: \tilde{X}_t \notin  \mbox{int}(\bfE) \} <T^*)=0$ when  $\tilde{X}$ satisfies  
\[
\tilde{X}_t=x+\int_0^t b(\tilde{X}_u)du+\int_0^t \sigma(\tilde{X}_u)\,dB_u,
\]
then there exists a unique strong solution, $X$, to (\ref{smb:e:sdeB}). Moreover, $X_{T^*}=z$. 
\end{theorem}
\begin{proof}
Let $X$ be a strong solution to (\ref{smb:e:sdeB}) on $(\Om, \cF, (\cF_t)_{t \in [0,T^*]}, P)$ and $T<T^*$.  Due to the form of $\bfE$ there exists an increasing sequence of open sets with compact closure, $U_n$, such that $\mbox{cl}(U_n) \subset \mbox{int}(\bfE)$ and $\mbox{int}(\bfE)=\cup_{n=1}^{\infty}U_n$.  It is a simple matter to check that the coefficients of (\ref{smb:e:sdeB}) satisfy the conditions of Theorem 5.3.7 of \cite{EK} on $[0,T]\times U_{n}$ for all $n \geq 1$. Thus, the conclusion of the theorem yields uniqueness upto $\tau_{n}:= T \wedge \inf\{t>0:X_t \notin U_{n}\}$. By taking limits first as  $n\rar \infty$ we obtain uniqueness until $T \wedge \tau$, where $\tau=\inf\{t>0: X_t \notin  \mbox{int}(\bfE) \}$. 

Let $\nu_m=\inf\{t\geq 0: h(t,X_t)\leq \frac{1}{m}\}$ and consider a measure $Q^{n,m}$ defined by 
\[
Q^{n,m}(E)= E\left[\frac{h(0,x)}{h\left(T \wedge \nu_m \wedge \tau_n, X_{T \wedge \nu_m \wedge \tau_n}\right)}\chf_E\right], \qquad E \in \cF_T.
\]
Due to the choice of stopping times $Q^{n,m}\sim P$. Moreover, 
\[
W_t = B_t +\int_0^{t \wedge \tau_n \nu_m }\sigma^*(X_u)\frac{(\nabla h(u,X_u))^*}{h(u,X_u)}du 
\]
is a $Q^{n,m}$-Brownian motion. Let $\tilde{X}$ be the unique strong solution of 
\be \label{smb:e:tildex}
\tilde{X}_t=x+\int_0^t b(\tilde{X}_u)du+\int_0^t \sigma(\tilde{X}_u)\,dW_u.
\ee
 Observe that  $X$ solves (\ref{smb:e:tildex}) under $Q^{n,m}$ until $T \wedge \nu_m \wedge \tau_n$, and thus $Q^{n,m}(X_{t \wedge \nu_m \wedge \tau_n}=\tilde{X}_{t \wedge \tilde{\nu}_m \wedge \tilde{\tau}_n}, t \in [0,T])=1$ by Theorem 5.3.7 of \cite{EK}, where $\tilde{\tau}_{n}:= T \wedge \inf\{t>0:\tilde{X}_t \notin U_{n}\}$, $\tilde{\nu}_m=\inf\{t\geq 0: h(t,\tilde{X}_t)\leq \frac{1}{m}\}$. Hence,
\bean
P(\tau_n \wedge \nu_m <T)&=&E^{Q^{n,m}}\left[\frac{h(T\wedge \tau_n \wedge \nu_m, X_{T \wedge \nu_m \wedge \tau_n})}{h(0,x)}\chf_{[\tau_n \wedge \nu_m <T]}\right]\\
&=&E^{Q^{n,m}}\left[\frac{h(T\wedge \tilde{\tau}_n \wedge \tilde{\nu}_m, \tilde{X}_{T \wedge \tilde{\nu}_m \wedge \tilde{\tau}_n})}{h(0,x)}\chf_{[\tilde{\tau}_n \wedge \tilde{\nu}_m <T]}\right]\\
&=&E^x\left[\frac{h(T\wedge \tilde{\tau}_n \wedge \tilde{\nu}_m, \tilde{X}_{T \wedge \tilde{\nu}_m \wedge \tilde{\tau}_n})}{h(0,x)}\chf_{[\tilde{\tau}_n \wedge \tilde{\nu}_m <T]}\right]\\
&=&E^x\left[\frac{h(T, \tilde{X}_{T })}{h(0,x)}\chf_{[\tilde{\tau}_n \wedge \tilde{\nu}_m <T]}\right]\\
&\leq & E^x\left[\frac{h(T, \tilde{X}_{T })}{h(0,x)}\left(\chf_{[\tilde{\tau}_n <T]} +\chf_{[\tilde{\nu}_m <T]}\right)\right],
\eean
where $P^x$ is the law of $\tilde{X}$. Note that the fourth equality is due to the fact that $(h(t,\tilde{X}_t))_{t \in [0,T]}$ is a martingale under $P^x$ as a consequence of Chapman-Kolmogorov equation. 

Next, observe that $ \lim_{m\rar \infty}\nu_m \geq \tau,\, P$-a.s. due to the strict positivity of $h$ on $\mbox{int}(\bfE)$. Similarly, $\lim_{m\rar \infty}\tilde{\nu}_m \geq \lim_{n \rar \infty} \tilde{\tau}_n,\, P^x$-a.s.. Furthermore, $\lim_{n \rar \infty} \tilde{\tau}_n>T,\, P^x$-a.s. as $\tilde{X}$ stays in the interior under $P^x$. Hence, taking the limits in the above yields $P(\tau<T)=0$ and, thus establishes the pathwise uniqueness upto $T$.  This together with the existence of a weak solution due to Theorem \ref{smb:th:MB} implies the existence of a unique strong solution via Corollary 5.3.23 in \cite{KS} on $[0,T^*)$ as $T$ was arbitrary.

Finally, since pathwise uniqueness implies uniqueness in law  we conclude that the law of $(X_t)_{t \in [0,T*)}$ is given by $P^{x \rar z}_{0 \rar T^*}$ obtained in Theorem \ref{smb:th:MB} because of the continuity of the weak solution. Thus, we can uniquely define $X_{T^*}=\lim_{t \rar T^*}X_t=z$.  
\end{proof}
\begin{remark} Using the arguments above we can also establish the the existence and uniqueness of a strong solution of the SDE in Theorem \ref{c:hsde} under the assumption that $b$ and $\sigma$ are locally Lipschitz. 
\end{remark}
We end this section with a result that shows that in the one-dimensional case the hypotheses of Assumption \ref{smb:a:potden} are not too restrictive.
\begin{proposition} \label{smb:l:ubdd} Suppose that Assumption \ref{smb:a:A} and the conditions of Example \ref{ex:1dd} are satisfied. Then, the conditions (\ref{smb:e:dualinf})-(\ref{smb:e:bddpotential}) hold.
\end{proposition}
\begin{proof} As the endpoints of $\bfE$ are inaccessible, the semi-group $P_t$ possesses a density $p(t,x,y)$ satisfying (\ref{smb:e:C-Kinf}) with respect to the speed measure, $m$, of the diffusion. Moreover, $p(t,x,y)=p(t,y,x)$ for all $t>0$ and $(x,y)\in (l,\infty)^2$ and is continuous on $(0,\infty)\times (l,\infty)\times (l,\infty)$ (see \cite{McKean56}). This, together with the continuity of $X$ yields (\ref{smb:e:dualinf}). 

Moreover, $u^{\alpha}(x,y)$ is symmetric for each $\alpha>0$ and is continuous on $(l,\infty)\times \bfE$ due to its construction and boundary behaviour (see (2.3) and Table 1 in \cite{McKean56}). 

It remains to show (\ref{smb:e:bddpotential}). If $y\in \bfE$ is a natural boundary, $u^{\alpha}(x,y)=0$ for all $\alpha>0$ and $x \in \bfE \backslash \{y\}$.  So, assume $y$ is not a natural boundary. 

First, note that for any continuous and bounded $f$, 
\be \label{smb:e:potlimitf}
\alpha \int_{\bfE} u^{\alpha}(x,z)f(z)m(dz)=\alpha\int_0^{\infty} e^{-\alpha t} E^xf(X_t)dt \rar f(x),
\ee
as $\alpha \rar \infty$ by the continuity of $X$.

 Next, let $K_1$ and $K_2$ be two disjoint closed and bounded intervals contained in $(l, \infty)$. It follows (see, e.g., Theorem 3.6.5  in \cite{MarRos}) from the strong Markov property of $X$ that for $x \in K_1$ and $z \in K_2$ that 
\[
u^{\alpha}(x,z)=E^x\left[e^{-\alpha \tau_2}\right]u^{\alpha}(X_{\tau_2},z),
\]
where $\tau_2=\inf\{t>0: X_t\in K_2\}$ and $X_{\tau_2}$ is deterministic and equals either the left or the right endpoint of $K_2$ depending on whether $x < z$ or not. Note that $X_{\tau_2}$ has the same value for all $x \in K_1$, and consequently
\[
u^{\alpha}(x,z) \leq u^{\alpha}(X_{\tau_2},z), \forall x\in K_1.
\]
Let $K_3$ be an arbitrary closed interval strictly contained in $K_2$. For any $f$ with a support in $K_3$, we have
\[
\sup_{x \in K_1}\int_0^{\infty} \alpha u^{\alpha}(x,z)f(z)m(dz)\leq \int_0^{\infty} \alpha u^{\alpha}(X_{\tau_2},z)f(z)m(dz).
\]
This implies in view of Fatou's lemma and (\ref{smb:e:potlimitf}) that whenever $x_n \rar x \in K_1$ and $\alpha_n \rar \infty$, we have
\be \label{smb:e:ualphlim}
\liminf_{n \rar \infty} \alpha_n u^{\alpha_n}(x_n,z)=0,
\ee
for $m$-a.a. $z$ in $K_2$. Since $m$ is equivalent to the Lebesgue measure in $(l, \infty)$, we have that the above holds for a.a. $z$ in $K_2$. 

 Now, suppose $K$ is a compact set that does not contain $y$ and there exists a convergent sequence $(x_n)$ in $K$ and a sequence of positive real numbers $(\alpha_n)$ such that $\lim_{n \rar \infty} \alpha_n =\infty$ as well as
 \be \label{smb:e:ubddpot}
 \lim_{n \rar \infty} \alpha_n u^{\alpha_n}(x_n,y)= \infty.
 \ee 
 
 In view of (\ref{smb:e:ualphlim}), there exists a $z \notin K\cup\{y\}$ such that $P^y(T_z <T_x)=1$ for all $x \in K$, and
\be \label{smb:e:ualphlim2}
 \liminf_{n \rar \infty} \alpha_n u^{\alpha_n}(x_n,z)=0.
 \ee
 On the other hand, using Theorem 3.6.5  in \cite{MarRos} once more and the fact that $u^{\alpha}$ is symmetric, we may write
 \[
 \limsup_{n \rar \infty} \frac{E^y\left[e^{-\alpha_n T_{x_n}}\right]}{E^z\left[e^{-\alpha_n T_{x_n}}\right]}=\limsup_{n \rar \infty}\frac{u^{\alpha_n}(x_n,y)}{u^{\alpha_n}(x_n,z)}=\infty,
 \]
where the value of the limit  follows from (\ref{smb:e:ubddpot}) and (\ref{smb:e:ualphlim2}). However, 
 \[
 \frac{E^y\left[e^{-\alpha_n T_{x_n}}\right]}{E^z\left[e^{-\alpha_n T_{x_n}}\right]}= E^y\left[e^{-\alpha_n T_{z}}\right]< 1
 \]
by the strong Markov property of $X$ since $P^y(T_z<T_{x_n})=1$.
\end{proof}

\bibliographystyle{siam}
\bibliography{ref}
\appendix
\section{Weak solutions and the local martingale problem} \label{s:app}
We will demonstrate in this section that under our standing assumptions the solution to the local martingale problem for $A$ defined in (\ref{e:At}) is equivalent to the weak solution of the following SDE:
\be \label{ch3_sde}
dX_t= b(t,X_t)\,dt + \sigma(t,X_t)dW_t
\ee
where $W$ is an $d$-dimensional Brownian motion and 
$\sigma$ is a $d\times d$ dispersion matrix such that
\[
a_{ij}(t,x)=\sum_{k=1}^r\sigma_{ik}(t,x)\sigma_{kj}(t,x).
\]
We will work under our standing assumption that $a$ and $b$ satisfy the conditions (1) and (2) of Assumption \ref{smb:a:A}. The connection between the weak solutions of  SDEs and the associated local martingale problems is well-known. Although it is usually observed under the slightly stronger condition that  $b$ is locally bounded, the proofs remain valid under our assumptions. Thus, we only give the statement of the theorem and the references for its proof.    

\begin{definition} A weak solution starting from $s$ of (\ref{ch3_sde}) is a triple
  $(X,W),\, (\Om,\cF, P),\, (\cF_t)_{t \geq 0}$, where
\begin{itemize}
\item[i)] $(\Om, \cF, P)$ is a probability space, $(\cF_t)$ is a
  filtration of sub-$\sigma$-fields of $\cF$, and $X$ is a process on $(\Om, \cF, P)$ with sample paths in $C([0,\infty),\bfE)$ such that $P(X_r=X_s, r \leq s)=1$;
\item[ii)] $W$ is $d$-dimensional  $(\cF_t)$-Brownian motion and $X=(X_t)$ is adapted to $(\bar{\cF}_t)$ where $\bar{\cF}_t$ is the completion of $\cF_t$ with $P$-null sets;
\item[iii)] $X$ and $W$ are such that
 \[
X_t=X_s+\int_s^{t}b(u,X_u)\,du + \int_s^t \sigma(u,X_u)\,dW_u, \, P\mbox{-a.s.}, \forall t\geq s.
\]
\end{itemize}
\end{definition}
The probability measure $\mu$ on $\sE$ defined by
$\mu(\Lambda)=P(X_s \in \Lambda)$ is called the {\em initial
  distribution} of the solution. 
  
 Observe that for any continuous process $Y$, the integrals $\int_s^t b(u,Y_u)\,du$ and $\int_s^t a(u,Y_u)\,du$ are well-defined  until $\tau_n=\inf\{t\geq s: |Y_t|\geq n\}$ for any $n \geq 1$. Due to the continuity of $Y$,  $\tau_n \rar \infty$  implying 
\be \label{smb:e:a:QV}
P\left(\int_s^t a(u,Y_u)\,du <\infty\right)=1.
\ee
Therefore, $\left(\int_s^t \sigma(u,Y_u)\,dW_u\right)_{t\geq s}$ is a well-defined continuous local martingale.  If, additionally, $Y$ is a weak solution of (\ref{ch3_sde}), it follows that
\[
P\left(\left|\int_s^t b(u,Y_u)\,du\right| <\infty\right)=1.
\]
Using the fact that $b$ is locally bounded from above or below, one can show that the above implies
\[
P\left(\int_s^t \left|b(u,Y_u)\right| \,du<\infty\right)=1, \qquad t \geq s,
\]
following the reasoning that led to  (\ref{smb:e:a:QV}). Thus, any weak solution is a semimartingale.

Equivalence of local martingale problem and weak solutions is summarised in the following theorem. Its  proof follows the lines that led to  Corollary 5.3.4 in \cite{EK}. Note that although  Corollary 5.3.4 in \cite{EK} assumes $a$ and $b$ are locally bounded, the proof therein applies under our assumptions as well.
\begin{theorem} \label{g:t:lm2ws} For any fixed $s \geq 0$ the existence of a solution $P^{s,\mu}$ to the local
  martingale problem for $(A, \mu)$ starting from $s$ is equivalent to the existence of a weak solution
  $(X,W),\, (\hat{\Om},\hat{\cF}, \hat{P}),\, (\hat{\cF}_t)$ to
  (\ref{ch3_sde}) starting from $s$ such that $\hat{P}(X_s\in \Lambda)=\mu(\Lambda)$ for any $\Lambda \in \sE$. The two solutions are related by $P^{s,\mu}=\hat{P}X^{-1}$;
  i.e. $X$ induces the measure $P^{s,\mu}$ on $(C(\bbR_+,\bfE),\sB)$. 
  
  Moreover, $P^{s,\mu}$ is unique if and only if the uniqueness in the sense of probability law holds for the solutions of  (\ref{ch3_sde}) starting from $s$ with the initial distribution $\mu$.
\end{theorem}
\section{Some technical results}
\begin{lemma} \label{smb:l:potential} Fix $x,z \in \bfE$ and $t>0$ such that $p(t,x,z)>0$. Define $M_s=p(t-s, X_s, z)$ for $s<t$. Then, $(M_s)_{s \in [0,t)}$ is a $P^x$-martingale. Moreover, if $(M_s)_{s \in [0,t)}$ is \cadlag, $P^x$-a.s.\footnote{Note that we can in fact always choose a \cadlag version as soon as we augment the natural filtration of $X$ with the universal null sets since $X$ is strong Markov. For a proof of this result see Theorem 4 and its Corollary in Section 2.3 of \cite{CW} and observe that although the result therein is proved for Feller processes its proof only uses the strong Markov property of a Feller process.}, and if (\ref{smb:e:dual}) is satisfied, then $M_t:=\lim_{s \rar t} M_s=0,\, P^x$-a.s. and $(M_s)_{ s \in [0,t]}$ is a $P^x$-supermartingale.
\end{lemma}

\begin{proof}
Observe that for any $0\leq s\leq u<t$ we have
\[
E^x[M_u|\cB_s]=\int_{\bfE}p(t-u, y, z)p(u-s, X_s, y)m(dy)=p(t-s,X_s,z)= M_s,
\]
and therefore $M$ is a martingale on $[0,t)$
and Theorem 1.3.15 in \cite{KS} yields that $M_t$ exists, non-negative, and $E^x[M_t]\leq E^x[M_0]$. 
Moreover, by Fatou's lemma we have 
\bean
E^x\left[M_t\chf_{B^c_{\frac{1}{n}}(z)}(X_t)\right]&\leq &E^x\left[\liminf_{u\rar t}M_u\chf_{B^c_{\frac{1}{n+1}}(z)}(X_u)\right]\leq\liminf_{u\rar t}E^x\left[M_u\chf_{B^c_{\frac{1}{n+1}}(z)}(X_u)\right]\\
&=&\liminf_{u\rar t}\int_{B^c_{\frac{1}{n+1}}(z)}p(t-u, y, z)p(u, x, y)m(dy)=0,
\eean
where the first inequality is due to the continuity of $X$, and the last equality is due to (\ref{smb:e:dual}). In view on non-negativity of $M_t\chf_{B^c_{\frac{1}{n}}(z)}(X_t)$, we have 
$$
 E^x\left[M_t\chf_{B^c_{\frac{1}{n}}(z)}(X_t)\right]=0.
$$
Monotone convergence theorem implies
$$
  E^x\left[M_t\right]=E^x\left[M_t\chf_{[X_t\neq z]}\right]=E^x\left[M_t\lim_{n\rar \infty}\chf_{B^c_{\frac{1}{n}}(z)}(X_t)\right]=\lim_{n\rar \infty} E^x\left[M_t\chf_{B^c_{\frac{1}{n}}(z)}(X_t)\right]=0,
$$
and since $M_t$ is non-negative, $M_t=0$, $P^x$-a.s..
\end{proof}
\begin{lemma}\label{smb:l:LT} Consider $\varphi : (0,\infty)\times (0,\infty) \mapsto [0,\infty) $. Then
\begin{itemize}
 \item[a.] If $\varphi(t,\cdot)$ is increasing and either 
    \begin{itemize}
      
      \item[(i)] 
      \[
        \lim_{\alpha \rar \infty} \alpha\int_0^{t} e^{-\alpha s }\varphi(t,s) ds =0,\,\forall t>0
      \]
       or 
       \item[(ii)] 
      \[
         \lim_{\alpha \rar \infty} \alpha\int_0^{\infty} \int_0^t e^{-\alpha s -\beta t}\varphi(t,s) dsdt =0,\, \forall \beta >0, 
      \]
      
      \end{itemize}
   then $\varphi(t,0):=\lim_{\delta\rar 0}\varphi(t,\delta)=0$ for almost every $t>0$.
 \item[b.] If there exists a constant $K$ such that $\varphi(t,\delta)<k$ for all $\leq\delta\leq t$, and $\lim_{\delta\rar 0}\varphi(t,\delta)=0$, then 
  \[
        \lim_{\alpha \rar \infty} \alpha\int_0^{t} e^{-\alpha s }\varphi(t,s) ds =0,
      \]
\end{itemize}
\end{lemma}

\begin{proof}
\begin{itemize}

\item[a.] First, observe that, since $\varphi(t,\cdot)$ is increasing and non-negative, we have 
$$
  0\leq \varphi(t,0)\leq \varphi(t,\delta), \, \forall\delta\geq 0 
$$ and therefore 
\begin{itemize}
      
      \item[(i)] 
      \[
        \varphi(t,0)=\lim_{\alpha \rar \infty} \alpha\int_0^{t} e^{-\alpha s }\varphi(t,0) ds\leq\lim_{\alpha \rar \infty} \alpha\int_0^{t} e^{-\alpha s }\varphi(t,s) ds =0.
      \]
  
       \item[(ii)] Due to Fatou's Lemma and the fact that $1-e^{ -\alpha t}$ is increasing in $\alpha$
      \bean
           0\leq\int_0^{\infty}e^{ -\beta t}\varphi(t,0) dt &=&\int_0^{\infty} \liminf_{\alpha \rar \infty} \left(1-e^{ -\alpha t}\right)e^{ -\beta t}\varphi(t,0) dt \\
           &\leq & \lim_{\alpha \rar \infty} \int_0^{\infty} \left(1-e^{ -\alpha t}\right)e^{ -\beta t}\varphi(t,0) dt\\
            &=& \lim_{\alpha \rar \infty} \alpha\int_0^{\infty} \int_0^t e^{-\alpha s -\beta t}\varphi(t,0) dsdt \\
         && \leq \lim_{\alpha \rar \infty} \alpha\int_0^{\infty} \int_0^t e^{-\alpha s -\beta t}\varphi(t,s) dsdt =0. 
      \eean
      
      Therefore, we have 
      $$
        \int_0^{\infty}e^{ -\beta t}\varphi(t,0) dt=0,
      $$
      which implies that $\varphi(t,0)=0$ for almost every $t>0$.
      
      \end{itemize}
  \item[b.] On the other hand, consider $\varepsilon>0$. As $\lim_{\delta\rar 0}\varphi(t,\delta)=0$, there exists $\delta>0$ such that $\varphi(t,s)<\varepsilon$ for all $s\leq \delta$. Then we will have 
  \bean
    0\leq \lim_{\alpha \rar \infty} \alpha\int_0^{t} e^{-\alpha s }\varphi(t,s) ds & \leq &\lim_{\alpha \rar \infty} \alpha\left[\int_{\delta}^{t} e^{-\alpha s }\varphi(t,s) ds+\varepsilon \int_0^{\delta} e^{-\alpha s } ds \right]\\
    & \leq &\lim_{\alpha \rar \infty}\left[K\left(e^{-\alpha \delta }-e^{-\alpha t }\right)+\varepsilon \left(1- e^{-\alpha \delta}\right)\right]=\varepsilon.
  \eean
  The conclusion follows due to arbitrariness of $\varepsilon$.
 \end{itemize}
\end{proof}
\end{document}